\documentclass[reqno,10pt]{amsart}
\usepackage[english]{babel}
\usepackage{amsthm}
\usepackage{amsmath}
\usepackage{amssymb}
\usepackage{amsfonts}
\usepackage{mathrsfs}
\usepackage{todonotes}
\usepackage{comment}
\usepackage{cancel}

\usepackage[utf8]{inputenc}
\usepackage{latexsym}
\usepackage{verbatim}

\usepackage[all]{xy}
\usepackage{dsfont}
\usepackage[bookmarksnumbered,colorlinks,urlcolor=blue]{hyperref}
\usepackage{graphicx}
\usepackage{graphics}
\usepackage{float}
\usepackage{enumerate}
\def\bb#1\eb{\textcolor{blue}
	{#1}} %
\def\br#1\er{Justificar
	{#1}} %
\newtheorem*{theorem*}{Theorem}
\newtheorem*{corollary*}{Corollary}
\newtheorem*{proposition*}{Proposition}
\newtheorem*{remark*}{Remark}

\def\br#1\er{Justificar{#1}} %
\def\bb#1\eb{\textcolor{blue}{#1}} %

\begin{document}
	
	\title[A lifting principle of curves under exponential-type maps]{A Lifting principle of curves under exponential-type maps}

	\author{Ivan P. Costa e Silva}
	\address{Department of Mathematics, Universidade Federal de Santa Catarina, 88.040-900
		Florianópolis-SC, Brazil.}
	\email{pontual.ivan@gmail.com}

	\author{José L. Flores}
	\address{Departamento de Álgebra, Geometría y Topología, Facultad de Ciencias, Universidad de Málaga, Campus Teatinos, 29071 Málaga, Spain}
	\email{floresj@uma.es}

		\begin{abstract}
		We develop a lifting theory for the exponential map of semi-Riemannian manifolds that overcomes the classical obstruction caused by its singularities. We show that every smooth path in the manifold admits, up to a nondecreasing reparametrization, a partial lift through the exponential map which is inextensible in its domain of definition. If the exponential map satisfies the path-continuation property—a natural topological condition—these lifts extend globally, yielding a general path-lifting theorem.
		
		This lifting approach yields new, alternative proofs of (generalizations of) a number of foundational results in semi-Riemannian geometry: the Hopf–Rinow theorem and Serre's classic theorem about multiplicity of connecting geodesics in the Riemannian case, as well as the Avez–Seifert theorem for globally hyperbolic spacetimes in Lorentzian geometry. More broadly, our results reveal the central role of the continuation property in obtaining geodesic connectivity across a wide range of semi-Riemannian geometries. This offers a unifying geometric principle that is complementary to the more traditional analytic, variational methods used in to investigate geodesic connectedness, and provides new insight into the structure of geodesics, both on geodesically complete and non-complete manifolds.

We also briefly point out how the lifting theory developed here can extend to more general flow-inducing maps on the tangent bundle other than the geodesic flow, suggesting broader geometric applicability beyond the exponential map.
	\end{abstract}

	\newtheorem{thm}{Theorem}[section]
	\newtheorem{prop}[thm]{Proposition}
	\newtheorem{lemma}[thm]{Lemma}
	\newtheorem{cor}[thm]{Corollary}
	\theoremstyle{definition}
	\newtheorem{defi}[thm]{Definition}
	\newtheorem{notation}[thm]{Notation}
	\newtheorem{exe}[thm]{Example}
	\newtheorem{conj}[thm]{Conjecture}
	\newtheorem{prob}[thm]{Problem}
	\newtheorem{rem}[thm]{Remark}
	\newtheorem{conv}[thm]{Convention}
	\newtheorem{crit}[thm]{Criterion}
	\newtheorem{claim}[thm]{Claim}

	\newcommand{\veq}{\mathrel{\rotatebox{90}{$\simeq$}}}
	\newcommand{\ben}{\begin{enumerate}}
		\newcommand{\een}{\end{enumerate}}
	
	\newcommand{\bit}{\begin{itemize}}
		\newcommand{\eit}{\end{itemize}}

	
	\maketitle
	
	\vspace*{-.5cm}
	
	\section{Introduction}\label{section0}

	Map-lifting techniques play a fundamental role across all areas of differential geometry and topology, particularly in the study of smooth maps between manifolds. Such techniques are frequently employed when analyzing the behavior of dynamical systems, investigating path-connectedness, or studying the existence and global properties of solutions to differential equations underlying geometric structures.
	
A particularly relevant instance for our purposes here is the notion of {\em path lifting} and its connection with the so-called {\em path-continuation} property, as investigated by F. Browder and W. Rheinboldt in the 1950s and 1960s \cite{Bro,Re}. Let us briefly recall the essential aspects of this framework. We shall consider throughout smooth maps $\mathcal{F}:N_1\rightarrow N_2$ between smooth (connected) manifolds $N_1$ and $N_2$, and use the term {\em path} to refer to continuous, piecewise smooth curves unless otherwise specified.

	A map $\mathcal{F}:N_1\rightarrow N_2$ is said to have the {\it path-lifting property} (\cite[Def. 2.3]{Re}) if for any path $\alpha:[0,1] \rightarrow N_2$ and any point $x_0 \in \mathcal{F}^{-1}(\alpha(0))$ there exists a path $\overline{\alpha}:[0,1] \rightarrow N_1$ such that 
	$$\overline{\alpha}(0) = x_0\quad \hbox{  and   }\quad \mathcal{F}\circ \overline{\alpha } = \alpha .$$
	A classic result in the theory of covering spaces ensures that any covering map\footnote{By \textit{covering map} here we always mean a {\it smooth covering map}, that is, it is onto and any $y\in N_2$ has an {\it evenly covered} neighborhood $V \ni y$, i.e., $\mathcal{F}^{-1}(V)\subset N_1$ is a disjoint union of open sets restricted to each of which $\mathcal{F}$ is a \textit{diffeomorphism} onto $V$.} possesses the path-lifting property. 
	
	A key contribution of Rheinboldt’s work \cite{Re} is the characterization of the path-lifting property in terms of a purely topological condition known as the path-continuation property \cite[Def. 2.2]{Re}. Specifically, a map 
	$\mathcal{F}:N_!\rightarrow N_2$ has the path-continuation property if, given any path $\alpha:[0,1]\rightarrow N_2$ and any continuous curve $\alpha:[0,b)\rightarrow N_1$, with $0<b\leq 1$, satisfying 
	$$\mathcal{F}\circ \overline{\alpha } = \alpha\mid_{[0,b)},$$
	there exists a sequence $(t_k)_{k\in \mathbb{N}}\subset [0,b)$ converging to $b$, such that the sequence $(\overline{\alpha}(t_k))_{k\in \mathbb{N}}$ converges in $N_1$. 
	This leads to the following simple but fundamental result \cite[Thm. 2.4]{Re}: {\em A local diffeomorphism $\mathcal{F}:N_1\rightarrow N_2$ has the path-lifting property if and only if it has the path-continuation property.} In particular, any smooth covering map $\mathcal{F}:N_1\rightarrow N_2$ satisfies {\em both} the path-lifting {\em and} the path-continuation properties.

	Our work is motivated by questions of geodesic connectedness in semi-Riemannian geometry, with a particular focus on the exponential map $\exp_p:\mathcal{D}\subset T_pM \rightarrow M$ of a semi-Riemannian manifold $(M,g)$.

	While our primary interest lies in semi-Riemannian geometry, the mechanism underlying our approach depends solely on the dynamical structure of flows on the tangent bundle. Indeed, the arguments extend naturally to maps associated with general vector fields on $TM$. For clarity, we formulate the main results in terms of the exponential map, postponing the more abstract formulation to Section \ref{abstractlifting}.
	
	A central challenge is that $\exp_p$ typically possesses singularities, so it may fail to be a local diffeomorphism everywhere. In such cases, path-lifting cannot be inferred from path-continuation alone. Nevertheless, in the Riemannian setting, geodesic connectedness as stated in the Hopf–Rinow theorem ensures surjectivity of the exponential map under geodesic completeness. As shown in \cite[Prop. 2.6]{CF}, this guarantees the path-continuation property even in the presence of singularities.
	
	These observations suggest that, even when singularities occur, exponential-type maps may retain essential features—such as surjectivity or the existence of lifted paths—by virtue of the path-continuation property.
	
	The main contribution of this paper is to show that the difficulties caused by singular points in the exponential map can be overcome. We demonstrate that every path in the manifold can be lifted as far as the domain of the exponential map allows. Furthermore, if the map satisfies the path-continuation property, the lift extends globally within the domain.
	
	The only adjustment required is that, instead of genuine lifts, we consider a \emph{quasi-lift}: a lift of some nondecreasing reparametrization of the original path (see Definition \ref{uh} for details).
	
	Our main result formalizes this principle for the semi-Riemannian exponential map (see Theorem \ref{maint}), although the argument applies more broadly:
	
	\smallskip
	
	\noindent {\bf Main Theorem.} {\em Let $\exp_p:\mathcal{D}\subset T_pM\rightarrow M$ be the exponential map at $p\in M$ for a semi-Riemannian manifold $(M,g)$. Then, any smooth regular curve in $M$ whose initial data lie in the image of $d\exp_p$ admits a partial quasi-lift that is either global or inextensible as a path. Moreover, if $\exp_p$ satisfies the continuation property, the quasi-lift can be chosen to be global.}

	\smallskip
	
	Our proof is quite technical. Therefore, we briefly summarize here the main steps and ideas of our approach.


Let $\mathcal{S}\subset \mathcal{D}(\subset T_p M)$ denote the set of singular points of the exponential map $\exp_p$ of $(M,g)$. Consider an arbitrary curve $\gamma:[0,1]\rightarrow M$, and suppose that it is a geodesic with respect to some auxiliary Riemannian metric $h$ on $M$, with $\gamma(0)=\exp_p(v_0)$ for some $v_0\in \mathcal{D} \setminus \mathcal{S}$. Since $v_0$ is a regular point of the exponential map, the initial velocity $\dot{\gamma}(0)$ uniquely determines a vector $w_0 \in T_{v_0}M$ such that the curve is initialized at $(v_0, w_0) \in T(\mathcal{D} \setminus \mathcal{S})$ via the relation $\dot{\gamma}(0) = (d\exp_p)_{v_0}(w_0)$.
Let $\exp_p^*(h)$ denote the pullback metric defined on $\mathcal{D}\setminus \mathcal{S}$, and let
$\overline{\alpha}:[0,l)\rightarrow \mathcal{D}\setminus\mathcal{S}$
be the maximal geodesic for $\exp_p^*(h)$ with initial conditions $\overline{\alpha}(0)=v_0$ and $\dot{\overline{\alpha}}(0)=w_0$. 

If $l>1$ then the composition $\alpha:=\exp_p\circ\overline{\alpha}$ agrees with $\gamma$ on $[0,1]$, and hence $\overline{\alpha}\mid_{[0,1]}$ provides a lift of $\gamma$.
 The difficulty arises when $\overline{\alpha}$ reaches the singular set $\mathcal{S}$ before $\gamma$ is fully traversed. To address this issue, we proceed as follows.
 
 Let $\mathcal{K} \subset \mathcal{D}$ be a compact neighborhood of $v_0$. 
 For each $\xi > 0$, we consider the non-degenerate metric $\overline{h}_{\xi}$ 
 (see (\ref{eq:Gxi}) and Proposition \ref{la1}) obtained 
 by pulling back the Sasaki metric 
 $h_S$ through a family of immersions $\{\Psi_\xi\}_{\xi>0}$ of $\mathcal{D}$ into the tangent bundle $TM$ that approach the horizontal distribution (defined as the $h_S$-orthogonal complement to the vertical distribution) as 
 $\xi \to 0$.

%
%

 We then consider the geodesic 
 $\overline{\alpha}_{\xi} : [0,1] \rightarrow \mathcal{K} \subset \mathcal{D}$ 
 for the metric $\overline{h}_\xi$, with initial conditions 
 $\overline{\alpha}_{\xi}(0) = v_0$ and $\dot{\overline{\alpha}}_{\xi}(0) = w_0$ 
 (Definition \ref{but}, Proposition \ref{uff}). 
 The resulting curve 
 $\alpha_{\xi} := \exp_p \circ \overline{\alpha}_{\xi}$ 
 need not be an $h$-geodesic and may differ from $\gamma$. 
 However, since the above immersions approach the horizontal distribution as 
 $\xi \to 0$, one expects that $\alpha_{\xi}$ provides an increasingly accurate 
 approximation of $\gamma$ as $\xi \to 0$.

A main difficulty in this construction is that the singularities of $\exp_p$ along conjugate directions lead to a loss of control at the level of first derivatives as $\xi \to 0$. 
In particular, the norm of $\dot{\overline{\alpha}}_\xi$ may blow up as the velocity approaches a singular direction, preventing a direct $C^1$ comparison between the curves $\alpha_\xi$ and the reference geodesic $\gamma$. 
To overcome this issue, 
we show that this effect is compensated at the level of the $L^1$-norm of the velocity: although $|\dot{\overline{\alpha}}_\xi|_{h_p}$ may become large, the set of parameters where this occurs becomes negligible as $\xi \to 0$, yielding a uniform bound on the $h_p$-length of $\overline{\alpha}_\xi$, independent of $\xi$ (Proposition \ref{l3}). As a consequence, despite the absence of derivative control, the curves remain close in the $C^0$ sense after reparametrization.

More precisely, this integral control 
allows us to choose 
a suitable reparametrization of $\overline{\alpha}_{\xi}$ to apply the Arzelà--Ascoli theorem and extract a uniformly convergent subsequence, yielding a limit curve $\overline{\alpha}$. 
Note however that, due precisely to this reparametrization, the limit curve may degenerate on certain subintervals; consequently, $\overline{\alpha}$ need not be a genuine lift of $\gamma$, but rather a partial quasi-lift in our sense.

The limit curve $\overline{\alpha}$ may be only a partial, rather than a global, quasi-lift of $\gamma$, because it may reach the boundary of $\mathcal{K}$ before its image under $\exp_p$ has fully traversed $\gamma$. Thus, the obstruction is not caused by a loss of control of $\exp_p\circ\overline{\alpha}_\xi$, which remains close to $\gamma$, but by the possibility that $\overline{\alpha}$ exits the region where the construction is carried out.

This phenomenon may still occur even when considering an exhausting sequence of compact subsets $\{\mathcal{K}_i\}$ of $\mathcal{D}$. In that case, the corresponding partial quasi-lifts may converge to a curve that is inextensible in $\mathcal{D}$, while its image under $\exp_p$ covers only a proper initial segment of $\gamma$.

To rule out this domain obstruction, we invoke the continuation property. This condition ensures the existence of a compact subset $\mathcal{K}\subset\mathcal{D}$ containing any piecewise smooth curve whose image under $\exp_p$ remains sufficiently controlled (in particular, sufficiently close to $\gamma$; see Proposition~\ref{1.3}). Consequently, the quasi-lift cannot leave the domain before $\exp_p\circ\overline{\alpha}$ has covered the entire curve $\gamma$, and we obtain a global quasi-lift as desired.


The significance of our theorem stems from the plethora of potential geometric applications of quasi-lifts of the exponential map, which naturally encode geodesic connectivity.
In fact, several new geometric results follow as immediate and transparent corollaries of our lifting principle, extending classical theorems that were originally proved using sophisticated variational or topological methods.
We summarize some of these results below.
\begin{itemize}	
	\item Extension of the geodesic-connectedness part of the Hopf–Rinow Theorem to arbitrary semi-Riemannian manifolds satisfying the continuation property (Theorem \ref{ñh}). In the Riemannian case, this yields a new and independent proof of a classic result on connectedness via minimal geodesics (Theorem \ref{tuss}). In the Lorentzian case, this result addresses a problem that goes back to the very origins of Lorentzian Geometry: finding a sharp and geometrically natural condition that ensures geodesic connectedness on Lorentzian manifolds. Indeed, it has long been known that even completeness or compactness are insufficient to guarantee such connectedness, as standard counterexamples show. We identify the continuation property as the sought-after condition, offering a conceptually clean and geometrically intrinsic criterion that fills this long-standing gap.
	
	\item Existence of infinitely many connecting geodesics between two points on any semi-Riemannian manifold under mild assumptions: the continuation property together with the non-properness of the exponential map (Theorem \ref{thr}). In the Riemannian case, these conditions are implied, for instance, by geodesic completeness and non-contractibility of the underlying manifold (Corollary \ref{AA}). This generalizes a classic result by Serre \cite{Se} concerning the existence of infinitely many geodesics connecting two points on a non-contractible Riemannian manifold ---a result originally established via deep variational methods and Morse theory. Our approach, by contrast, obviates the need for such analytic techniques and relies instead on a purely geometric lifting principle grounded in the path-continuation property.
	
	\item Extension of the Avez–Seifert theorem in Lorentzian geometry: the standard global hyperbolicity assumption is replaced by the weaker {\em causal continuation property} (Definition \ref{def1}, Theorem \ref{thmclave3}). This substantially broadens the scope of that foundational result in Lorentzian geometry, which among other perks has found broad applications in general relativity, where the existence of such geodesics carries a well-established physical interpretation as light rays and freely-falling particles with mass, playing a crucial role for examples in most of the all-important singularity theorems.
\end{itemize}

	Last but not least, our result is crucial to overcome the challenges posed by the e\-xis\-ten\-ce of the so-called {\em self-conjugate points} when addressing the notion of {\em geodesic homotopy} — a notion first introduced in \cite{CFH2} as a tool for finding closed geodesics. In fact, as de\-mons\-tra\-ted in \cite{CF}, geodesic homotopy can, with the results established here, be applied in full generality, without the need to exclude self-conjugate points by hand. We are thus able to either produce new results on closed geodesics or enhance existing ones (see \cite{CF} for further details).
	
%
%

The remainder of this paper is structured as follows. In Section \ref{section1}, we introduce some key definitions and present the main theorem. Section \ref{dos} is devoted to yielding several preliminary results that form the backbone of the subsequent analysis. The proof of the main theorem is given in Section \ref{tres}. To illustrate the scope and significance of this result, Section \ref{s5} derives a series of novel and immediate consequences associated with the problem of geodesic connectedness on semi-Riemannian manifolds, with particular emphasis on the Riemannian and Lorentzian cases discussed earlier. Finally, in Section \ref{abstractlifting}, we show that the quasi-lifting framework developed here extends well beyond the exponential map, and applies to the natural maps associated with arbitrary vector fields on $TM$.
		
	\section{Definitions and statement of the main result}\label{section1}
	
	Henceforth, all smooth manifolds throughout the paper are assumed to be connected.
	
	\begin{defi}\label{uh} Let $\mathcal{F}:N_1\rightarrow N_2$ denote a smooth map. 
A continuous curve $\sigma:[0,c]\rightarrow N_1$ with $0< c<\infty$ 
		is a {\em partial quasi-lift (through $\mathcal{F}$)}  of the piecewise smooth regular curve $\gamma:[0,1]\rightarrow N_2$ 
		if  there exists 
		a continuous, nondecreasing, surjective function $\chi:[0,c]\rightarrow [0,d]$
		so that $\gamma \circ\chi= \mathcal{F}\circ\sigma$ for some $0<d\leq 1$. If $d=1$ we simply say that $\sigma$ is a {\em (global) quasi-lift}. 
	\end{defi}	
	
	The reparametrizing function $\chi:[0,c]\rightarrow [0, d]$ in the previous definition may be constant on certain subintervals. In such cases, it is not possible to obtain an \emph{ordinary lift}---i.e., a quasi-lift with $\chi(t) \equiv t$ for all $t \in [0,1]$ and $c = d = 1$---simply by reparametrizing conveniently $\sigma$. 
This obstruction reflects the geometric limitations imposed by the presence of singularities in the lifting map: although the image of the path can still be recovered through a quasi-lift, the parameterization cannot always be preserved. What can be achieved 
is that $\chi$ becomes an \emph{Aztec step function}---that is, a piecewise smooth function $\chi : [0, c] \rightarrow [0,1]$, $c\geq 1$, whose derivative alternates between $0$ and $1$. This structure captures the essential feature of quasi-lifts: the lifting curve progresses along the base path precisely when allowed by the geometry, and pauses otherwise. To this end, let $\chi:[0,c]\to[0,d]$ be a continuous, non-decreasing function. 
One can construct a continuous function $\zeta:[0,\tilde{c}]\to[0,c]$, strictly increasing on the subintervals where $\chi$ is strictly increasing, such that the function $\tilde{\chi} := \chi \circ \zeta : [0,\tilde{c}] \longrightarrow [0,d]$ satisfies $\frac{d}{dx} \tilde{\chi}(x) = 1$ on those subintervals. The construction proceeds as follows: consider the quantile function associated with $\chi$, $\zeta(y) := \inf \{ x\in[0,c] : \chi(x) \ge y \}$, $y\in[0,d]$, and adjust the domain linearly on the intervals where $\chi$ is constant to obtain a parametrization that is continuous and strictly increasing on the non-constant segments. 
By construction, $\zeta$ is continuous, and $\tilde{\chi}$ has unit derivative on the intervals where $\chi$ is not constant (for further details, 
see for instance, \cite[Chapter 7]{Ru}).


	\begin{defi}\label{wpdeff} 
		Let $\mathcal{F}:N_1\rightarrow N_2$ be a smooth map between smooth manifolds. We say that $\mathcal{F}$ has the 
		{\em continuation property} if for any piecewise smooth 
		curve $\gamma:[0,1]\rightarrow N_2$ and any continuous curve $\sigma:[0,b)\subset [0,1]\rightarrow N_1$ such that $\mathcal{F}\circ\sigma=\gamma\mid_{[0,b)}$, there exists a sequence $(t_k)_{k\in \mathbb{N}}$ in $[0,b)$ converging to $b$ for which the sequence $\{\sigma(t_k)\}_{k\in\mathbb{N}}$ converges on $N_1$.
	\end{defi}	
	Due to the key role played by the continuation property in this work, we establish here a couple of alternative characterizations of this notion. 	
	\begin{prop}\label{1.3} Let $\mathcal{F}:N_1\rightarrow N_2$ be a smooth map. The following statements are equivalent.
		\begin{itemize}
			\item[(i)] $\mathcal{F}$ has the continuation property.
			\item[(ii)] $\mathcal{F}$ is {\em weakly proper}, i.e.  
		any 
		continuous curve $\sigma:[a,b)\rightarrow N_1$ with $-\infty< a<b\leq 1$ 
		such that $\mathcal{F}\circ \sigma$ is right-extendible in $N_2$ has image contained inside a compact set of $N_1$. (This notion was first introduced in \cite{CF}.)\footnote{Recall that a smooth map $\mathcal{F}:N_1\rightarrow N_2$ between smooth manifolds is called {\em proper} if the preimage by $\mathcal{F}$ of any compact set in $N_2$ is compact in $N_1$. So, just as the name suggests, any proper map is weakly proper.}
		\item[(iii)] Given $q\in N_1$ and $\ell>0$, there exists a compact neighborhood $\mathcal{C}$ of $q$ in $N_1$ such that any 
			continuous curve $\sigma:[a,b)\rightarrow N_1$ with $\sigma(a)=q$ for which ${\rm length}_h(\mathcal{F}\circ \sigma)\leq \ell$ has its image contained in $\mathcal{C}$, where $h$ is any auxiliary complete Riemannian metric on $N_2$.
	\end{itemize}	
\end{prop}	

\noindent {\it Proof.} The proof of the equivalence (i)$\Leftrightarrow$(ii) is given infor 
\cite[Proposition 4.3]{CFH}\footnote{The result \cite[Proposition 4.3]{CFH} actually refers to the particular case $\mathcal{F}=\exp_p$, $N_1=\mathcal{D}\subset T_p M$ and $N_2=M$, but the same proof is valid in the general case.}. The implication (iii)$\Rightarrow$(ii) is straightforward. Accordingly, we shall focus only on the implication (ii)$\Rightarrow$(iii), to consider which we fix a background complete Riemannian metric $h$ on $N_2$, a point $q\in N_1$ and a number $\ell>0$.

Assume, by contradiction, that $(ii)$ holds, together with the existence of continuous curves $\sigma_m:[a,b]\rightarrow N_1$ with $\sigma_m(a)=q$, 
and a nondecreasing sequence of numbers $\{b_j\}\subset [a,b]$ such that 
\begin{equation}\label{aa1}
	{\rm length}_h(\mathcal{F}\circ\sigma_{m})\leq \ell\qquad\hbox{and}\qquad \sigma_m(b_j)\in \mathcal{C}_j\setminus \mathcal{C}_{j-1}\quad \forall\hbox{$1\leq j\leq m$,}
\end{equation}
where $\{\mathcal{C}_j\}$ is a sequence of compact neighborhoods of $q$ with $\mathcal{C}_{j-1}\subset \mathring{\mathcal{C}}_j$ for all $j$ and $\cup_{j=1}^{\infty}\mathcal{C}_j=N_1$. 
We can assume without loss of generality that all $\sigma_m$ are piecewise smooth. 


Let $j=1$ and consider the sequence $\{\sigma_m(b_1)\}\subset \mathcal{C}_1$. Since $\mathcal{F}$ is smooth and $\mathcal{C}_1$ is compact, this sequence admits a subsequence  $\{\sigma_{m^1_k}(b_1)\}$, $m^1_k\geq 2$, and a sequence of smooth curves $\{\tau^1_{kk'}\}$, $k<k'$, connecting $\sigma_{m^1_k}(b_1)$ with $\sigma_{m^1_{k'}}(b_1)$ such that
\[
{\rm length}_h(\mathcal{F}\circ\tau^1_{kk'})<\frac{1}{2}\quad\forall k<k'.
\]
For $i=2$ the sequence $\{\sigma_{m^1_k}(b_2)\}\subset \mathcal{C}_2$ admits a subsequence  $\{\sigma_{m^2_k}(b_2)\}$, $m^2_k\geq 3$, and a sequence of smooth curves $\{\tau^2_{kk'}\}$, $k<k'$, connecting $\sigma_{m^2_k}(b_2)$ with $\sigma_{m^2_{k'}}(b_2)$ such that
\[
{\rm length}_h(\mathcal{F}\circ\tau^2_{kk'})<\frac{1}{2^2}\quad\forall k<k'.
\]
Proceeding in this way by induction, we construct sequences $\{m^1_k\}\supset\{m^2_k\}\supset\cdots\supset\{m^j_k\}\supset\cdots$, with $m^j_k\geq j+1$, and sequences of curves $\{\tau^j_{kk'}\}$, $k<k'$, connecting $\sigma_{m^j_k}(b_j)$ with $\sigma_{m^j_{k'}}(b_j)$, such that 
\[
{\rm length}_h(\mathcal{F}\circ\tau^j_{kk'})<\frac{1}{2^j}\quad\forall k<k'.
\]
Finally, replace the original sequence $\{\sigma_m\}$ by the diagonal subsequence $\{\sigma_n:=\sigma_{m^n_n}\}$ (and consequently, the original sequence $\{b_j\}$ by $\{b_i:=b_{j^i_i}\}$) in order to ensure that it satisifies, in addition to (\ref{aa1}), the condition
\begin{equation}\label{aa2}
	{\rm length}_h(\mathcal{F}\circ\tau^i_{nn'})<\frac{1}{2^i}\quad\forall i\leq n<n'.
\end{equation}
for smooth curves $\tau^i_{nn'}$ connecting $\sigma_{n}(b_i)$ with $\sigma_{n'}(b_i)$.

Next, define for each $i$, 
\[
l_i:={\rm inf}\{{\rm length}_h((\mathcal{F}\circ\sigma_n)\mid_{[b_i,b_{i+1}]}):\; n\geq i+1\},
\]
and choose some $n_i\geq i+1$ such that 
\begin{equation}\label{aa3}
	{\rm length}_h((\mathcal{F}\circ\sigma_{n_i})\mid_{[b_i,b_{i+1}]})<l_i+\frac{1}{2^i}.
\end{equation}
Define $\alpha_{i}:=\sigma_{n_i}\mid_{[b_{i},b_{i+1}]}$ and denote by $\alpha_{i\,i+1}$ the curve $\tau^{i+1}_{n_i n_{i+1}}$ which connects $\sigma_{n_i}(b_{i+1})$ with $\sigma_{n_{i+1}}(b_{i+1})$.   
Finally, consider the piecewise smooth curve $\alpha$ in $N_1$ obtained by making the following countably infinite concatenation:
\[
\alpha:=\alpha_1 * \alpha_{12} * \alpha_{2} * \alpha_{23} * \alpha_{3} \cdots
\] 
By construction, $\alpha$ passes through $\sigma_{n_1}(0)=q$ and $\sigma_{n_i}(b_i)\not\in \mathcal{C}_{i-1}$ for all $i$ (recall (\ref{aa1})). So, $\alpha$ is not contained in any compact set of $N_1$. On the other hand, from (\ref{aa2}) and (\ref{aa3}),
\[
\begin{array}{rl}
	{\rm length}_h(\mathcal{F}\circ\alpha) & ={\rm length}_h(\mathcal{F}\circ\alpha_1)+{\rm length}_h(\mathcal{F}\circ\alpha_{12})+{\rm length}_h(\mathcal{F}\circ\alpha_{2})
	+\cdots \\ & <(l_1+\frac{1}{2^1})+\frac{1}{2^2}+(l_2+\frac{1}{2^2})+\frac{1}{2^3}+\cdots \\ & < \sum_{i=1}^\infty l_i+2\sum_{i=1}^{\infty}\frac{1}{2^i},
\end{array}
\] 
where, by definition of $l_i$, we have $\sum_{i=1}^{k-1} l_i\leq {\rm length}_h((\mathcal{F}\circ\sigma_k)\mid_{[0,b_k]})\leq \ell$ for all $k$ (recall (\ref{aa1})). Therefore,
\[
{\rm length}_h(\mathcal{F}\circ\alpha)<\sum_{i=1}^\infty l_i+2\sum_{i=1}^{\infty}\frac{1}{2^i}\leq \ell+2\cdot 1<\infty,
\] 
in contradiction with the weak properness of $\mathcal{F}$. \qed

\smallskip

As the inclusion map $i:\mathbb{R}\rightarrow \mathbb{R}^2$, $i(x)=(x,0)$ shows, the continuation property does not guarantee in general the existence of a quasi-lift for any smooth path.
In order to explore a context where this implication is nevertheless satisfied, we will restrict our attention to exponential-type maps.

We are now ready to state the main result of this note.

\begin{thm}\label{maint} 
	Let $\exp_p:\mathcal{D}\subset T_pM\rightarrow M$ be the exponential map at a point $p$ of a semi-Riemannian manifold $(M,g)$, with $\mathcal{D}$ its maximal domain. Then, any smooth regular map $\gamma:[0,b]\rightarrow M$ with $(d\exp_p)(v_0,w_0)=(\gamma(0),\dot{\gamma}(0))$ for some $(v_0,w_0)\in T\mathcal{D}$ admits a partial quasi-lift with initial data $(v_0,w_0)$ which is either global or inextensible as a path. Moreover, if $\exp_p$ satisfies the continuation property, the quasi-lift can be chosen to be global.
\end{thm}

\begin{rem}
	Under the continuation property, the assumption on the initial data
	is automatically satisfied for any smooth regular curve.
	Indeed, given a path $\gamma$, one may extend it backwards so that it starts at $p$
	(recall that $(d\exp_p)_0$ is non-singular), apply the theorem to this extended curve,
	and then restrict the resulting quasi-lift to the original interval.
\end{rem}

\section{Preliminary technical results}\label{dos}

	Throughout this section, the symbols $C_{\mathcal{K}}$ and $\xi_{\mathcal K}$ denote positive constants whose values may change from line to line. When necessary, they are redefined (for instance, by taking maxima or minima) in order to satisfy the required properties while preserving the previous ones.

Let $(M,g)$ be a semi-Riemannian manifold, $p\in M$, and denote by
\[
\exp_p : \mathcal{D} \subset T_p M \longrightarrow M
\]
the exponential map at $p\in M$, with $\mathcal{D}$ its maximal domain. Let $h$ be a Riemannian metric on $M$, and let $h_S$ denote the Sasaki metric on $TM$ induced by $h$. For $\xi>0$, define the smooth map
\[
\Psi_\xi : \mathcal{D} \longrightarrow TM, 
\qquad 
\Psi_\xi(v) := \big(\exp_p(v),(d\exp_p)_v(\xi v)\big),
\]
and set
\begin{equation}\label{eq:Gxi}
	\overline{h}_\xi := \Psi_\xi^*(h_S).
\end{equation}

\begin{prop}\label{la1}
	For each $\xi>0$, the map $\Psi_\xi$ is an immersion. In particular,
	$\overline{h}_\xi$ is a Riemannian metric on $\mathcal{D}$.
\end{prop}

\begin{proof}
	Given $v\in \mathcal{D}$, 
	$w\in T_v(T_pM)\setminus\{0\}$, let \(v(s)\subset T_pM\) be a smooth curve with
	\[
	v(0)=v,
	\qquad
	v'(0)=w.
	\]
	This curve induces a geodesic variation
	\[
	\Gamma(s,t)=\exp_p\big(t\,v(s)\big).
	\]
	The associated variational field
	\[
	J(t)=\left.\frac{\partial}{\partial s}\right|_{s=0}\Gamma(s,t)
	\]
	is a Jacobi field along the geodesic \(\gamma(t)=\Gamma(0,t)=\exp_p(tv)\) satisfying
	\[
	J(0)=0,
	\qquad
	J'(0)=w\neq 0.
	\]
	By definition,
	\[
	(d\Psi_\xi)_v(w)
	=
	\left.\frac{D}{ds}\right|_{s=0}
	\Psi_\xi(v(s)).
	\]
	Writing explicitly,
	\[
	\Psi_\xi(v(s))
	=
	\Big(
	\exp_p(v(s)),
	(d\exp_p)_{v(s)}\big(\xi v(s)\big)
	\Big),
	\]
	and differentiating with respect to $s$, we obtain
	\[
	(d\Psi_\xi)_v(w)
	=
	\Big(
	J(1),
	\xi\, J'(1)
	\Big)
	\in T_{\Psi_\xi(v)}(TM),
	\]
	where we have used the standard identification given by the horizontal--vertical
	decomposition of $T(TM)$.

If $(d\Psi_\xi)_v(w)=0$, then $J(1)=0$ and $J'(1)=0$. 
By uniqueness of Jacobi fields, it follows that $J\equiv 0$. 
Hence $w=J'(0)=0$, a contradiction.
Therefore $(d\Psi_\xi)_v(w)\neq 0$ for every 
	$w\neq 0$, and $\Psi_\xi$ is an immersion for every $\xi>0$.
	
	Finally, since $h_S$ is a Riemannian metric on $TM$ and $\Psi_\xi$ is an immersion, 
	its pullback $\overline{h}_\xi = \Psi_\xi^*(h_S)$
	is a positive definite symmetric $(0,2)$–tensor on $\mathcal{D}$. 
	Thus, $\overline{h}_\xi$ is a Riemannian metric, as required.
\end{proof}

Observe that when $\xi=0$ the map $\Psi_\xi$ formally reduces to $v \mapsto (\exp_p(v),0)$,
so that the tensor $\overline h_\xi=\Psi_\xi^*(h_S)$ can be viewed as a
regularized version of the pullback tensor $(\exp_p)^*h$ on $\mathcal D$.
The latter may become degenerate at points where the 
$d\exp_p$ loses rank, that is, at conjugate directions. The additional
velocity component in the definition of $\Psi_\xi$ restores positive
definiteness while preserving the geometry of the geodesic flow.

The following definition introduces a key notion for constructing path-lifting via a convergence procedure.	

\begin{defi}\label{but}
Let $\gamma:[0,b]\to M$ be an $h$-geodesic, and fix some $\xi>0$.
A smooth curve $\overline{\alpha}_\xi:I=[0,l]\to \mathcal{D}$, $0<l\le b$, is called a $\xi$-{\em regularized lift (through $\exp_p$)} of $\gamma$ if it is a $\overline{h}_\xi$-geodesic satisfying $\exp_p(\overline{\alpha}_\xi(0))=\gamma(0)$ and
$(d\exp_p)_{\overline{\alpha}_\xi(0)}\big(\dot{\overline{\alpha}}_\xi(0)\big)=\dot{\gamma}(0)$. 
\end{defi}

Let $\overline{\alpha}_\xi$ be a $\xi$-regularized lift contained in a compact subset $\mathcal{K}\subset \mathcal{D}$, with initial data $\overline{\alpha}_\xi(0)=v_0$ and $\dot{\overline{\alpha}}_\xi(0)=w_0$, and set $\alpha_\xi := \exp_p \circ \overline{\alpha}_\xi$. Then $\alpha_\xi$ is not, in general, an $h$--geodesic. Indeed, by construction,
\[
\overline h_\xi = \Psi_\xi^*(h_S),
\qquad
\Psi_\xi(v) = \big(\exp_p(v),(d\exp_p)_v(\xi v)\big),
\]
so that $\Psi_\xi\circ\overline{\alpha}_\xi$ is a geodesic of the Sasaki metric $h_S$ on $TM$. Consequently, its horizontal component $\alpha_\xi$ satisfies
\begin{equation}\label{rru}
	\frac{D^h}{ds}\dot{\alpha}_\xi
	=
	R(\dot{\alpha}_\xi,V)\dot{\alpha}_\xi,
\end{equation}
where $R$ is the curvature tensor of $h$, and 
\[
V(s)=(d\exp_p)_{\overline{\alpha}_\xi(s)}(\xi \overline{\alpha}_\xi(s)).
\]
Observe that the velocity of $\Psi_\xi\circ\overline{\alpha}_\xi$ splits as
\[
\frac{d}{ds}\Psi_\xi(\overline{\alpha}_\xi(s))
=
\left(\dot{\alpha}_\xi(s),\, \xi\,\frac{D^h}{ds}(d\exp_p)_{\overline{\alpha}_\xi(s)}(\overline{\alpha}_\xi(s))\right).
\]
So, taking into account that $\Psi_\xi\circ\overline{\alpha}_\xi$ is a geodesic of $(TM,h_S)$, its speed is constant:
\[
\left|\frac{d}{ds}\Psi_\xi(\overline{\alpha}_\xi(s))\right|_{h_S}
=
\left|\frac{d}{ds}\Psi_\xi(\overline{\alpha}_\xi(0))\right|_{h_S}.
\]
Therefore,
\[
\begin{array}{rl}
|\dot{\alpha}_\xi(s)|_h^2 & \leq
|\dot{\alpha}_\xi(s)|_h^2
+
\left|\xi\,\frac{D^h}{ds}(d\exp_p)_{\overline{\alpha}_\xi(s)}(\overline{\alpha}_\xi(s))\right|_h^2=\left|\frac{d}{ds}\Psi_\xi(\overline{\alpha}_\xi(s))\right|_{h_S}^2  \\ & =\left|\frac{d}{ds}\Psi_\xi(\overline{\alpha}_\xi(0))\right|_{h_S}^2 =|\dot{\alpha}_\xi(0)|_h^2
+
\left|\xi\,\frac{D^h}{ds}(d\exp_p)_{\overline{\alpha}_\xi(0)}(\overline{\alpha}_\xi(0))\right|_h^2.
\end{array}
\]
Since $\dot{\alpha}_\xi(0)=\dot{\gamma}(0)$, $\overline{\alpha}_\xi(0)=v_0$ and $\dot{\overline{\alpha}}_\xi(0)=w_0$, the last term in the previous formula is bounded independently of $\xi$. Consequently,
\begin{equation}\label{cinco}
	|\dot{\alpha}_\xi(s)|_h \quad \text{is bounded independently of } \xi>0.
\end{equation}
On the other hand, since $\overline{\alpha}_\xi([0,l])\subset \mathcal{K}$ and $\mathcal{K}$ is compact, the differential of $\exp_p$ is uniformly bounded on $\mathcal{K}$, and thus,
\begin{equation}\label{vis}
|V(s)|_h \le C'\,\xi\,|\dot{\overline{\alpha}}_\xi(s)|_{h_p},\quad \hbox{for some constant $C'>0$.}
\end{equation}
Putting together \eqref{rru}, (\ref{cinco}) and (\ref{vis}), we can conclude as follows:
\begin{prop}\label{eq:alpha-acc'} Let $\mathcal{K}\subset \mathcal{D}$ be a compact set with $(d\exp_p)(v_0,w_0)=(\gamma(0),\dot{\gamma}(0))$  for some $(v_0,w_0)\in T\mathring{\mathcal{K}}$. Then, there exists some constant $C>0$ independent of $\xi$ such that for any $\xi$-regularized lift $\overline{\alpha}_\xi$ contained in $\mathcal{K}$, with initial data $\overline{\alpha}_\xi(0)=v_0$ and $\dot{\overline{\alpha}}_\xi(0)=w_0$, the following inequality holds:
\begin{equation}\label{eq:alpha-acc'}
	\left(\frac{d}{ds}|\dot{\alpha}_\xi(s)|_h
	\le\right)
	\left|\frac{D^h}{ds}\dot{\alpha}_\xi(s)\right|_{h}\le 
	C\,\xi\,|\dot{\overline{\alpha}}_\xi(s)|_{h_p}\quad\hbox{for any $s\in [0,l]$.}
\end{equation}
\end{prop}


\medskip	
	
	We now establish a natural existence result for $\xi$-regularized lifts, showing that such lifts can always be constructed under reasonable assumptions. 
	
	\begin{prop}\label{uff} Let $\mathcal{K}\subset \mathcal{D}$ be a compact set, and let  $\gamma:[0,b]\rightarrow M$ be an $h$-geodesic with $(d\exp_p)(v_0,w_0)= (\gamma(0),\dot{\gamma}(0))$ for some $(v_0,w_0)\in T\mathring{\mathcal{K}}$.
		Then, for any $\xi>0$ there exists a $\xi$-regularized lift $\overline{\alpha}_\xi:[0,l]\rightarrow\mathcal{K}$ of $\gamma$ with $\overline{\alpha}_\xi(0)=v_0$, $\dot{\overline{\alpha}}_\xi(0)=w_0$ and 
		$0<l\leq b$. 
		Moreover, if $l<b$ then ${\rm Im}\,\overline{\alpha}_\xi\not\subset \mathring{\mathcal{K}}$.
	\end{prop}

	\begin{proof} 
		Let $\hat{\alpha}_\xi:[0,\hat{l})\rightarrow \mathcal{D}$ be the maximal $\overline{h}_\xi$-geodesic 
	in $\mathcal{D}$ 
	with initial conditions $\hat{\alpha}_\xi(0)=v_0$, $\dot{\hat{\alpha}}_\xi(0)=w_0$. 
	Then, either $\hat{l}>b$ or $\hat{l}\leq b$ and some restriction
	$\hat{\alpha}_\xi\mid_{[0,l)}$, $0<l<\hat{l}$, is also right-inextendible in $\mathring{\mathcal{K}}$. In both cases
	we obtain the required $\xi$-regularized lift $\overline{\alpha}_\xi$ just by conveniently restricting $\hat{\alpha}_\xi$. 
	\end{proof}

The next result will be crucial to ensure, together with the Ascoli-Arzelá theorem, the existence of a partial limit for a sequence of $\xi_n$-regularized lifts of $\gamma$ as $\xi_n\rightarrow 0$. But, first, we need a technical lemma.

\begin{lemma}\label{claim}
	Let $\overline\alpha_\xi$ be a geodesic of $(\mathcal D,\overline h_\xi)$ defined on $[0,l]$,
	and set
	\[
	a(s):=\overline{h}_{\xi}(\overline{\alpha}_\xi(s),\dot{\overline{\alpha}}_\xi(s))\quad\hbox{for all $s\in [0,l]$.}
	\]
	Then, there exist $\varepsilon_0, C_0, \xi_0>0$ 
	such that, for any $0<\xi\leq\xi_0$, 
	\[
		|\dot{a}(s)|
			\geq \varepsilon_0|\dot{\overline{\alpha}}_\xi(s)|_{h_p} \quad\hbox{on $I_{C_0}:=\{s\in [0,l]:\, |\dot{\overline{\alpha}}_\xi(s)|_{h_p}>C_0\}$.}
	\]
\end{lemma}

\begin{proof} 
	Denote by $\bar{J}_{\xi}(t)\equiv \bar{J}_{\xi}(t,s)$, 
	$s\in [0,l]$, the vector field on the curve $t\in [0,1]\mapsto t\,\overline{\alpha}_{\xi}(s)$ given by the variation $\overline{\Gamma}(t,s):=t\,\overline{\alpha}_{\xi}(s)$, i.e.
	\[
	\bar{J}_{\xi}(t):=\frac{d}{ds}\overline{\Gamma}(t,s)=t\dot{\overline{\alpha}}_{\xi}(s)\quad\hbox{on $[0,l]$}.
	\]
	Clearly,
	\[
	\bar{J}_{\xi}(0)=0,\qquad
	\bar{J}_{\xi}(1)=\dot{\overline{\alpha}}_{\xi}(s)
	\quad\hbox{on $[0,l]$}. 
	\]
	
	\medskip
	
	\noindent {\sc Claim.} There exist constants $\overline{\varepsilon}_0,\delta_0>0$ such that, for any $\xi>0$ sufficiently small, 
	\[
	0<|\bar{J}_{\xi}(1)|_{\overline h_\xi}
	<\delta_0|\dot{\overline\alpha}_{\xi}(s)|_{h_p}
	\quad\Longrightarrow\quad
	\left|\frac{d}{dt}\Big|_{t=1}|\bar{J}_{\xi}(t)|_{\overline h_\xi}\right|
	\ge \overline{\varepsilon}_0 |\dot{\overline\alpha}_{\xi}(s)|_{h_p}.
	\]
	
	\medskip

	\medskip
	
	\noindent {\em Proof of Claim.}
	Arguing by contradiction, there exist sequences $\xi_n\to 0$ and $s_n$
	such that $\bar{J}_n(t)\equiv\bar{J}_{\xi_n}(t,s_n)$ and $\overline{\alpha}_n\equiv \overline{\alpha}_{\xi_n}$ satisfy
	\begin{equation}\label{que}
	\frac{|\bar{J}_n(1)|_{\overline h_{\xi_n}}}
	{|\dot{\overline\alpha}_n(s_n)|_{h_p}}\to0,
	\qquad
	\frac{\left|\frac{d}{dt}\big|_{1}|\bar{J}_n(t)|_{\overline h_{\xi_n}}\right|}
	{|\dot{\overline\alpha}_n(s_n)|_{h_p}}\to0.
	\end{equation}
	Passing to a subsequence,
	\[
	(\overline\alpha_n(s_n),
	\dot{\overline\alpha}_n(s_n)/|\dot{\overline\alpha_n}(s_n)|_{h_p})
	\to (v_*,w_*)
	\in\hat T\mathcal D .
	\]
	Set
	\[
	\begin{array}{ll}
	\Gamma_n(t,s):=\exp_p\left(t\left(\overline{\alpha}_n(s_n)+s\frac{\dot{\overline{\alpha}}_n(s_n)}{|\dot{\overline{\alpha}}_n(s_n)|_{h_p}}\right)\right), &
	J_n(t)\equiv J_n(t,s):=\partial_s \Gamma_n(t,s)
	\\ \\
	 \Gamma(t,s)=\exp_p\big(t(v_*+s w_*)\big), & J(t)\equiv J(t,s):=\partial_s \Gamma(t,s).
	 \end{array}
	\]
	Clearly,
	\[
	\Gamma_n \to \Gamma,
	\quad
	J_n \to J\quad\hbox{on compact sets.}
	\]
On one hand we have that \[
	J(0)=0,\qquad
	\frac{D^h}{dt}J(0)=w_*\neq 0,
	\]
	and thus, $J$ is non-trivial. On the other hand, the limits (\ref{que}) translate into
	\[
	J_n(1) \to 0,
	\qquad
	\frac{d}{dt}\big|_{1}|J_n(t)|_h \to 0,
	\]
	which imply
	\[
	J(1)=0,\qquad
	\frac{D^h}{dt}\mid_1J(t)=0,
	\]
	forcing $J$ being trivial, a contradiction. So, the claim holds.

	\medskip
	
	Next, there exists
	$C_0>0$ such that
	\[
	|\bar{J}_\xi(1)|_{\overline{h}_{\xi}}=|\dot{\overline{\alpha}}_\xi(s)|_{\overline{h}_{\xi}}=c_0<c_0 C_0^{-1}|\dot{\overline{\alpha}}_\xi(s)|_{h_p}<\delta_0|\dot{\overline{\alpha}}_\xi(s)|_{h_p}\quad\hbox{on $I_{C_0}$.} 
\]
Therefore, if we take $0<\varepsilon_0\leq  c_0\overline{\varepsilon}_0$, then
\begin{equation}\label{aa3'}
	\begin{array}{c}	
		\left|\overline{h}_{\xi}(\frac{D^{\overline{h}_\xi}}{ds}\overline{\alpha}_\xi(s),\dot{\overline{\alpha}}_\xi(s))\right|=\left|\overline{h}_{\xi}(\frac{D^{\overline{h}_\xi}}{dt}\mid_1 \bar{J}_\xi(t),\bar{J}_\xi(1))\right| \\ =|\bar{J}_\xi(1)|_{\overline{h}_{\xi}}\,\left|\frac{d}{dt}\mid_1 |J_\xi(t)|_{\overline{h}_{\xi}}\right| \stackrel{{\scriptsize {\rm Claim}}}{\geq} c_0\overline{\varepsilon}_0|\dot{\overline{\alpha}}_\xi(s)|_{h_p}\geq\varepsilon_0 |\dot{\overline{\alpha}_\xi}(s)|_{h_p}
		\quad\hbox{on $I_{C_0}$,}
	\end{array}
\end{equation}
where we have used in the first equality of previous formulas the identities:
\[
\begin{array}{c}
	\frac{D^{\overline{h}_\xi}}{ds}\overline{\alpha}_\xi(s)
	=\frac{D^{\overline{h}_\xi}}{ds}\frac{d}{dt}\mid_1 (t\overline{\alpha}_\xi(s)) =\frac{D^{\overline{h}_\xi}}{dt}\mid_1\frac{d}{ds}(t\overline{\alpha}_\xi(s)) 
	=
	\frac{D^{\overline{h}_\xi}}{dt}\mid_1\frac{d}{ds}\overline{\Gamma}(t,s)=\frac{D^{\overline{h}_\xi}}{dt}\mid_1 \bar{J}_\xi(t).
\end{array}
\]
\[
\dot{\overline{\alpha}}_\xi(s)=\frac{d}{ds}\overline{\alpha}_\xi(s)=\frac{d}{ds}\overline{\Gamma}(1,s)=\bar{J}_\xi(1),
\]
Summarizing,
\[
\begin{array}{c}
	|\dot{a}(s)|
	=\left|\frac{d}{ds}\overline{h}_{\xi}(\overline{\alpha}_\xi(s),\dot{\overline{\alpha}}_\xi(s))\right|
	=|\overline{h}_{\xi}(D^{\overline{h}_\xi}/ds\,\overline{\alpha}_\xi(s),\dot{\overline{\alpha}}_\xi(s))|\stackrel{(\ref{aa3'})}{\geq} \varepsilon_0|\dot{\overline{\alpha}}_\xi(s)|_{h_p}\;\,\hbox{on $I_{C_0}$.} 	\end{array}
\]
This 
concludes the proof. 
\end{proof}

\begin{prop}\label{l3} Let $\mathcal{K}\subset \mathcal{D}$ be a compact set. 
Given an $h$-geodesic $\gamma:[0,b]\rightarrow M$, 
there exist $\xi_{\mathcal K}, \Lambda_{\mathcal K}>0$ 
such that, for any $\xi$-regularized lift $\overline{\alpha}_{\xi}:I=[0,l]\rightarrow \mathcal{K}$ of $\gamma$, 
with
$0<\xi\leq\xi_{\mathcal{K}}$ and $0<l\leq b$, the following inequality holds:
\[
\int_{0}^l|\dot{\overline{\alpha}}_{\xi}(s)|_{h_p}ds<\Lambda_{\mathcal K}.
\]
\end{prop}

\begin{proof} Assume by contradiction the existence of a sequence $\{\overline{\alpha}_{n}\}$ of
$\xi_{n}$-regularized lifts
$\overline{\alpha}_n:[0,l_n]\rightarrow \mathcal{K}$ of $\gamma$ with $0<l_n\leq b$ and $\xi_n\searrow 0$ such that \begin{equation}\label{sq}
\int_{0}^{l_n}|\dot{\overline{\alpha}}_n(s)|_{h_p}\rightarrow\infty.
\end{equation}
Consider the smooth functions (previously introduced in Lemma \ref{claim})
\begin{equation}\label{uunoo}
a_n(s)=\overline{h}_{\xi_n}(\overline{\alpha}_n(s),\dot{\overline{\alpha}}_n(s)),
\end{equation}
and observe that the following identity holds:
\begin{equation}\label{tip}
\int_{s_n^-}^{s_n^+}\dot{a}_n(s)ds=a_n(s_n^+)-a_n(s_n^-),\;\quad 0\leq s_n^-<s_n^+\leq l_n.
\end{equation}
On the one hand, we know that the curves
\[
s\in [0,l_n]\mapsto \overline{\alpha}_n(s)
\]
are $\overline{h}_{\xi_n}$-geodesics on $[0,l_n]$, and thus,
\begin{equation}\label{jun'}
c_n:=|\dot{\overline{\alpha}}_n(0)|_{\overline{h}_{\xi_n}}=|\dot{\overline{\alpha}}_n(s)|_{\overline{h}_{\xi_n}}
\quad\hbox{for all $s\in [0,l_n]$.}
\end{equation}
Since $\xi_n\rightarrow 0$, the costants $c_n$ are bounded above. On the other hand, since $\overline{\alpha}_n$ remain in the compact set $\mathcal{K}\subset \mathcal{D}$ independently of $n$, and $\xi_n\rightarrow 0$, 
	\begin{equation}\label{try} 
		|\overline{\alpha}_n(s)|_{\overline{h}_{\xi_n}}\quad\hbox{ 
			has an upper bound on $[0,l_n]$ independent of $n$.}
	\end{equation}
	Taking into account (\ref{jun'}) and (\ref{try}) in the definition (\ref{uunoo}), we deduce that 
	\begin{equation}\label{yu}
		\Delta a_n:=a_n(s_n^+)-a_n(s_n^-)\quad\hbox{is bounded.} 
	\end{equation}
		Therefore, combining (\ref{tip}) and (\ref{yu}),
		the contradiction will follow if we prove that the left-hand side of (\ref{tip}) is unbounded.

To that end, recall that Lemma \ref{claim} provides constants $\varepsilon_0, \xi_0, C_0>0$ such that
\begin{equation}\label{ñoco''}
|\dot{a}_n(s)|\geq \varepsilon_0 |\dot{\overline{\alpha}}_n(s)|_{h_p}\quad\hbox{whenever $|\dot{\overline{\alpha}}_n(s)|_{h_p}>C_0$ and $0<\xi_n\leq\xi_0$.}
\end{equation}
Hence,
\[
\begin{aligned}
	\int_{I_n^{C_0}}|\dot{a}_n(s)|
	& \stackrel{(\ref{ñoco''})}{\ge} \varepsilon_0\int_{I_n^{C_0}}|\dot{\overline{\alpha}}_n(s)|_{h_p} \\
	&= \varepsilon_0\int_{0}^{l_n}|\dot{\overline{\alpha}}_n(s)|_{h_p}-\varepsilon_0\int_{[0,l_n]\setminus I_n^{C_0}}|\dot{\overline{\alpha}}_n(s)|_{h_p} \\
	&\ge \varepsilon_0\int_{0}^{l_n}|\dot{\overline{\alpha}}_n(s)|_{h_p}-\varepsilon_0\, C_0\, l_n \stackrel{(\ref{sq})}{\longrightarrow} \infty.
\end{aligned}
\]
Recall also that $I_n^{C_0}$ is open. Hence it is a disjoint union of connected components,
\[
I_n^{C_0} = \bigcup_{j}(\sigma_{n,j}^-,\sigma_{n,j}^+),
\]
and its complement
\[
[0,l_n]\setminus I_n^{C_0} = \{s\in[0,l_n]:|\dot{\overline{\alpha}}_n(s)|_{h_p}\le C_0\}
\]
is closed. The connected components of both sets alternate along $[0,l_n]$. Let $J_n$ be a connected component of $[0,l_n]\setminus I_n^{C_0}$ such that $\dot a_n$ changes sign across $J_n$. By continuity, $\dot a_n$ vanishes at some point of $J_n$. 

To conclude, we use the following estimate, whose proof is deferred.

\medskip

\noindent {\sc Key estimate.} There exists $\delta>0$ such that any connected 
component $J_n$ of $[0,l_n]\setminus I_n^{C_0}$ containing a zero of $\dot a_n$
satisfies $\mathrm{diam}(J_n)\ge \delta$.

\medskip

Assuming this estimate, the conclusion follows immediately. Indeed, since $\dot{a}_n$ is continuous, each connected component of $I_n^{C_0}$ is contained in an interval where $\dot{a}_n$ has constant sign (recall (\ref{ñoco''})). Moreover, between two connected components of $I_n^{C_0}$ with different signs for $\dot a_n$ there must exist a component $J_n$ of the complement of $I_n^{C_0}$ containing a zero of $\dot{a}_n$, hence of diameter $\ge \delta$ (according to the key estimate). Since $l_n\le b$, there can be only finitely many such sign changes. Hence there exists a maximal interval $(s_n^-,s_n^+)\subset [0,l_n]$ on which $\dot{a}_n$ has constant sign and
\[
\int_{s_n^-}^{s_n^+} |\dot{a}_n(s)|\,ds \longrightarrow +\infty.
\]
On this interval $\dot{a}_n$ does not change sign, so
\[
\Delta a_n = a_n(s_n^+) - a_n(s_n^-) = \int_{s_n^-}^{s_n^+} \dot{a}_n(s)\, ds \longrightarrow \pm \infty,
\]
which contradicts the boundedness claimed at \eqref{yu}, as required.

\medskip

\noindent\textit{Proof of the key estimate.}
Using the metric compatibility of the Levi--Civita connection, we compute
\[
\dot{a}_n
=
h\Big(\frac{D^h}{ds}(d\exp_p)_{\overline{\alpha}_n}(\overline{\alpha}_n), \dot{\alpha}_n\Big)
+
h\Big((d\exp_p)_{\overline{\alpha}_n}(\overline{\alpha}_n), \frac{D^h}{ds}\dot{\alpha}_n\Big).
\]
Since $(d\exp_p)_{\overline\alpha_n}(\overline\alpha_n)$ is uniformly bounded on $J_n$, the estimation (\ref{eq:alpha-acc'}) ensures that the second term is $O(\xi_n)$.

We now turn to the first term. Writing the covariant derivative in local coordinates, we have
\[
\frac{D^h}{ds}Y
=
\frac{dY}{ds}
+
\Gamma(Y,\dot{\alpha}_n),
\quad\text{where } Y=(d\exp_p)_{\overline{\alpha}_n}(\overline{\alpha}_n),
\]
and $\Gamma$ denotes the Christoffel symbols of $h$. By the chain rule,
\[
\frac{dY}{ds}
=
(d^2\exp_p)_{\overline{\alpha}_n}
(\dot{\overline{\alpha}}_n,\overline{\alpha}_n)
+
(d\exp_p)_{\overline{\alpha}_n}(\dot{\overline{\alpha}}_n).
\]
Hence,
\[
\frac{D^h}{ds}Y
=
(d^2\exp_p)_{\overline{\alpha}_n}
(\dot{\overline{\alpha}}_n,\overline{\alpha}_n)
+
(d\exp_p)_{\overline{\alpha}_n}(\dot{\overline{\alpha}}_n)
+
\Gamma(Y,\dot{\alpha}_n).
\]
Since $\overline\alpha_n(J_n)$ and $\alpha_n(J_n)$ lie in compact subsets, the first and second derivatives of $\exp_p$ are uniformly bounded. Moreover, $\dot{\overline\alpha}_n$ and $\dot{\alpha}_n$ are uniformly bounded on $J_n$. It follows that $\frac{D^h}{ds}Y$ is uniformly bounded on $J_n$. Consequently,
\[
\Big|h\Big(\frac{D^h}{ds}(d\exp_p)_{\overline{\alpha}_n}(\overline{\alpha}_n), \dot{\alpha}_n\Big)\Big|
=\Big|h\Big(\frac{D^h}{ds}Y, \dot{\alpha}_n\Big)\Big|
\le \lambda
\quad\text{on }J_n,
\]
for some constant $\lambda>0$ independent of $n$, and therefore
\[
|\dot{a}_n(s)| \le \lambda' \quad \text{on } J_n\quad\hbox{for some $\lambda'>0$.}
\]

Since $|\dot a_n| \le \lambda'$, the function $a_n$ can increase at most at rate $\lambda'$. Therefore, it cannot reach the value $\epsilon_0 C_0$ before a time interval of length at least
\[
\delta := \frac{\epsilon_0 C_0}{\lambda'}.
\]
Summarizing, every connected component $J_n$ of $[0,l_n]\setminus I_n^{C_0}$ containing a zero of $\dot{a}_n$ has diameter at least $\delta>0$, independent of $n$, as required.
\end{proof}

		\medskip
		
		Next, we consider the following minor improvement of \cite[Prop. 7.9, Sect. 7.2]{BE}. 
		\begin{prop}\label{t}
			Suppose that $F=(F_1,\ldots,F_m)$ and $G=(G_1,\ldots,G_m)$ are continuous functions defined on a common domain $D\subset  {\mathbb R}\times {\mathbb R}^m$, and suppose that $F$ satisfies the Lipschitz condition
			\[
			\|F(s,z)-F(s,z')\|_2\leq L\|z-z'\|_2\qquad\forall (s,z),\, (s,z')\in D.
			\]
			Let $z(s)=(z_1(s),\ldots,z_m(s))$, $z'(s)=(z'_1(s),\ldots,z'_m(s))$ be solutions for $a\leq s\leq b$ of the differential equations
			\[
			\frac{dz}{ds}=F(s,z)\quad\hbox{and}\quad \frac{dz'}{ds}=G(s,z'),
			\]
			respectively. If there exist $\kappa_1(s), \kappa_2\geq 0$ such that  $\|F(s,z)-G(s,z)\|_2\leq \kappa_1(s)+\kappa_2$ for all $(s,z)\in D$ with $a\leq s\leq b$,
			the following inequality holds for all $a\leq s\leq b$:
			\[
				\|z(s)-z'(s)\|_2\leq \left(\|z(a)-z'(a)\|_2+\int_a^b\kappa_1(s)ds\right) e^{L(s-a)}+\frac{\kappa_2}{L}(e^{L(s-a)}-1).
			\] 
		\end{prop} 
		
		\begin{proof}
			Set
			\[
			e(s):=z(s)-z'(s).
			\]
			Subtracting the two differential equations gives
			\[
			\dot e(s)
			=
			F(s,z(s)) - G(s,z'(s)).
			\]
			Adding and subtracting $F(s,z'(s))$, we obtain
			\[
			\dot e(s)
			=
			F(s,z(s)) - F(s,z'(s))
			+
			F(s,z'(s)) - G(s,z'(s)).
			\]
			Taking Euclidean norms and applying the triangle inequality yields
			\[
			\|\dot e(s)\|_2
			\le
			\|F(s,z(s)) - F(s,z'(s))\|_2
			+
			\|F(s,z'(s)) - G(s,z'(s))\|_2.
			\]
			By the Lipschitz assumption on $F$,
			\[
			\|F(s,z(s)) - F(s,z'(s))\|_2
			\le
			L \|e(s)\|_2.
			\]
			Moreover, by hypothesis,
			\[
			\|F(s,z) - G(s,z)\|_2
			\le
			\kappa_1(s) + \kappa_2
			\quad
			\text{for all } (s,z)\in D.
			\]
			Hence,
			\[
			\|\dot e(s)\|_2
			\le
			L\|e(s)\|_2
			+
			\kappa_1(s)
			+
			\kappa_2.
			\]
			Integrating from $a$ to $s\in[a,b]$ gives
			\[
			\|e(s)\|_2
			\le
			\|e(a)\|_2
			+
			\int_a^s L\|e(\tau)\|_2\,d\tau
			+
			\int_a^s (\kappa_1(\tau)+\kappa_2)\,d\tau.
			\]
			An application of the Gronwall inequality in integral form then yields
			\[
			\|e(s)\|_2
			\le
			\|e(a)\|_2 e^{L(s-a)}
			+
			\int_a^s (\kappa_1(\tau)+\kappa_2)\,e^{L(s-\tau)}\,d\tau.
			\]
			Splitting the integral, we obtain
			\[
			\|e(s)\|_2
			\le
			\|e(a)\|_2 e^{L(s-a)}
			+
			\int_a^s \kappa_1(\tau)e^{L(s-\tau)}\,d\tau
			+
			\int_a^s \kappa_2 e^{L(s-\tau)}\,d\tau.
			\]
			Since $e^{L(s-\tau)} \le e^{L(s-a)}$ for $\tau\in[a,s]$, the second term satisfies
			\[
			\int_a^s \kappa_1(\tau)e^{L(s-\tau)}\,d\tau
			\le
			e^{L(s-a)}\int_a^s \kappa_1(\tau)\,d\tau.
			\]
			Moreover,
			\[
			\int_a^s \kappa_2 e^{L(s-\tau)}\,d\tau
			=
			\frac{\kappa_2}{L}\left(e^{L(s-a)}-1\right).
			\]
			Combining these estimates yields
			\[
			\|z(s)-z'(s)\|_2
			\le
			\left(
			\|z(a)-z'(a)\|_2
			+
			\int_a^s \kappa_1(\tau)\,d\tau
			\right)
			e^{L(s-a)}
			+
			\frac{\kappa_2}{L}
			\left(e^{L(s-a)}-1\right),
			\]
			which is the desired estimate.
		\end{proof}

		\medskip

		Given an arbitrary coordinate chart $(U,x=(x_1,\ldots,x_n))$ for $M$, we shall obtain an associated coordinate chart $(TU,z=(x_1,\ldots,x_n,y_{1},\ldots,y_{n}))$
		for $TM$ as follows. Let $\partial/\partial x_1,\ldots,\partial/\partial x_n$ be the basis vector fields defined on $U$ by the local coordinates $x=(x_1,\ldots,x_n)$. Given $u\in T_p M$ for $p\in U$, we may write $u=\sum_{i=1}^n y_i\frac{\partial}{\partial x_i}\mid_p$. Then $z(p,u)$ is defined to be $z(p,u)=(x(p),y(u))=(x_1(p),\ldots,x_n(p),y_1(u),\ldots,y_n(u))$. 
		These coordinate charts may then be used to define Euclidean coordinates distances on $U$ and $TU$. Explicitly, given $(p,u), (q,v)\in TU$, set
		\[
		\|p-q\|_2=\left( \sum_{i=1}^n (x_i(p)-x_i(q))^2 \right)^{1/2}
		\]
		\[
		\|(p,u)-(q,v)\|_2=\left( \sum_{i=1}^{n} (x_i(p)-x_i(q))^2 +\sum_{i=1}^{n} (y_i(u)-y_i(v))^2\right)^{1/2}\qquad \hbox{resp.}
		\] 
		The geodesic equations expressed locally with respect to the coordinate chart $(TU,z)$ for $(M,h)$ are given by
		\[
		\frac{dz_{l}}{ds}=z_{l+n},\qquad \frac{dz_{l+n}}{ds}=-\Gamma^l_{jk}z_{j+n}z_{k+n},\quad 1\leq l,j,k\leq n,
		\] 
		where the functions $\Gamma^l_{jk}$ denote the Christoffel symbols for $h$ (which of course depend on $z_i$, $i=1,\ldots,n$), and the Einstein summation convention has been employed. We wish to apply Proposition \ref{t} to the geodesics of $(M,h)$ in $U$. Using the notation therein, we can identify $TU$ with a subset of ${\mathbb R}^{2n}$ using the coordinate chart $(TU,z)$, and define $F_h(s,z)\equiv F_h(z)$ by
		\begin{equation}\label{zx}
			(F_h)_l(z)=z_{l+n},\qquad 
			(F_h)_{l+n}(z)=-\Gamma^l_{jk}z_{j+n}z_{k+n},\quad 1\leq l,j,k\leq n,
		\end{equation}
		for $z=(z_1,\ldots,z_{2n})\in {\mathbb R}^{2n}$. Of course, the components of the function $F_h$ coincide with the components of the  
		geodesic spray $X_h$ on $(TU,h)$.
		
	\smallskip
	
	\begin{prop}\label{propnew}
		Let $\mathcal{K}\subset \mathcal{D}$ be a compact set with $(d\exp_p)(v_0,w_0)=(\gamma(0),\dot{\gamma}(0))$  for some $(v_0,w_0)\in T\mathring{\mathcal{K}}$. Then, there exists some constant $C>0$ independent of $\xi$ such that for any $\xi$-regularized lift $\overline{\alpha}_\xi$ contained in $\mathcal{K}$, with initial data $\overline{\alpha}_\xi(0)=v_0$ and $\dot{\overline{\alpha}}_\xi(0)=w_0$, the following inequality holds:
				\[
			\|(d\exp_p)_*F_{\overline h_\xi}(\overline{\alpha}_\xi,\dot{\overline{\alpha}}_\xi)
			-
			F_h\circ(d\exp_p)(\overline{\alpha}_\xi,\dot{\overline{\alpha}}_\xi)\|_{2}
			\le C\,\xi\,|\dot{\overline{\alpha}}_\xi|_{h_p}\quad\hbox{on $[0,l]$}.
		\]
	\end{prop}
	
	\begin{proof}
		By construction, $\overline{\alpha}_\xi$ is a geodesic of the Riemannian metric $\overline h_\xi$ on $\mathcal{D}$, and therefore
		\[
		F_{\overline h_\xi}(\overline{\alpha}_\xi(s),\dot{\overline{\alpha}}_\xi(s)) = \left(\dot{\overline{\alpha}}_\xi(s),\frac{D^{\overline h_\xi}}{ds}\dot{\overline{\alpha}}_\xi(s)\right)=(\dot{\overline{\alpha}}_\xi(s),0).
		\]
		Consider its horizontal projection $\alpha_\xi := \exp_p \circ \overline{\alpha}_\xi$. The pushforward of the $\overline h_\xi$-spray under $\exp_p$ reads
		\[
		(d\exp_p)_* F_{\overline h_\xi} (\overline{\alpha}_\xi(s), \dot{\overline{\alpha}}_\xi(s))
		=
		\left(\dot{\alpha}_\xi(s), 0\right) \in T_{\dot{\alpha}_\xi(s)} TM,
		\]
		On the other hand, the $h$-spray $F_h$ satisfies
		\[
		F_h\circ (d\exp_p)(\overline{\alpha}_\xi(s), \dot{\overline{\alpha}}_\xi(s)) = \left(\dot{\alpha}_\xi(s), \frac{D^h}{ds}\dot{\alpha}_\xi(s)\right) \in T_{\dot{\alpha}_\xi(s)} TM.
		\]
		Subtracting, we get
		\[
		(d\exp_p)_* F_{\overline h_\xi} (\overline{\alpha}_\xi(s), \dot{\overline{\alpha}}_\xi(s)) - F_h\circ (d\exp_p)(\overline{\alpha}_\xi(s), \dot{\overline{\alpha}}_\xi(s)) = 
		\left(0, -\frac{D^h}{ds}\dot{\alpha}_\xi(s)\right).
		\]
		But, according to (\ref{eq:alpha-acc'}),
		\[
		\left|\frac{D^h}{ds} \dot{\alpha}_\xi(s) \right|_h \le C\, \xi\, |\dot{\overline{\alpha}}_\xi(s)|_{h_p}.
		\]
		Therefore,
		\[
		\|(d\exp_p)_* F_{\overline h_\xi}(\overline{\alpha}_\xi(s), \dot{\overline{\alpha}}_\xi(s)) - F_h \circ (d\exp_p)(\overline{\alpha}_\xi(s), \dot{\overline{\alpha}}_\xi(s))\|_{2} \le C\, \xi\,|\dot{\overline{\alpha}}_\xi|_{h_p},
		\]
		as claimed.
	\end{proof}

		\section{Proof of the main theorem}\label{tres}
	
		Let $\gamma:I=[0,b]\rightarrow M$ be an $h$-geodesic for some Riemannian metric $h$ on $M$ such that $(d\exp_p)(v_0,w_0)=(\gamma(0),\dot{\gamma}(0))$ for some $(v_0,w_0)\in T\mathcal{D}$. Assume also that $\gamma(I)$ is contained in a relatively compact open coordinate chart $(U,x)$ of $M$.
		
Consider the associated coordinate chart $(TU,z=(x_{1},\ldots,x_{n},y_{1},\ldots,y_{n}))$ for $TM$ described in the previous section, and denote by $z^\sigma\equiv(x^\sigma,y^\sigma)\equiv(x\circ\dot{\sigma},y\circ\dot{\sigma})$ the coordinates associated to any curve $\sigma$ in $U$. 

Since $z^\gamma(I)$ is a compact subset of the open set $z(TU)\subset {\mathbb R}^{2n}$, there exists some $0<\epsilon_0<1$ small enough such that the closure of the set of points of $\mathbb{R}^{2n}$ whose $\|\cdot\|_2$-distance to $z^{\gamma}(s)$ is smaller than $\epsilon_0$ for some $s\in I$ is contained in $z(TU)$, i.e.
		\[
		K:=\cup_{s\in [0,b]}\overline{B_{\epsilon_0}(z^{\gamma}(s))}\subset z(TU).
		\]
Let $F_h$ be the ($C^1$) function given by (\ref{zx}) for $(TU,z)$, 
	and denote by $L$ a Lipschitz constant for $F_h$ on the compact set $K\subset z(TU)$.
		
	Fix a compact set $\mathcal{K}\subset \mathcal{D}$ such that $(v_0,w_0)\in T1\mathring{\mathcal{K}}$. By Propositions \ref{l3} and \ref{propnew} there exist constants $\Lambda_{\mathcal{K}}, \xi_{\mathcal{K}}, C>0$ such that, 
	for any $\xi$-regularized lift $\overline{\alpha}_\xi:[0,l]\rightarrow \mathcal{K}$ of $\gamma$, with
	$0<\xi\leq\xi_{\mathcal{K}}$ and $0<l\leq b$, 
	the following inequalities hold:
	\begin{equation}\label{rri}
		\int_{0}^l|\dot{\overline{\alpha}}_\xi(s)|_{h_p}ds<\Lambda_{\mathcal{K}}.
	\end{equation}
	\begin{equation}\label{ch}
	\|(d\exp_p)_*F_{\overline h_\xi}(\overline{\alpha}_\xi,\dot{\overline{\alpha}}_\xi)
	-
	F_h\circ(d\exp_p)(\overline{\alpha}_\xi,\dot{\overline{\alpha}}_\xi)\|_{2}
	\le C\,\xi\,|\dot{\overline{\alpha}}_\xi|_{h_p}.
	\end{equation}	 
	
	\smallskip
	
	The proof of Theorem \ref{maint} is essentially based on the following result and the subsequent corollary:
	
	\begin{prop}\label{previous}
		
		Let $0<\delta<\epsilon_0 e^{-Lb}$ and let $\mathcal{K}\subset \mathcal{D}$ be the compact set fixed above. Then there exists a $\xi$-regularized lift $\overline{\alpha}_\xi:[0,l]\to \mathcal{K}$ of $\gamma$, with $0\le \xi<\delta$ and $0<l\le b$, such that $\overline{\alpha}_\xi(0)=v_0$ and $\dot{\overline{\alpha}}(0)=w_0$, which is inextensible in $\mathcal{K}$ if $l<b$, and satisfies
		\[
		\|x^{\alpha_\xi}(s)-x^{\gamma}(s)\|_2<\delta e^{Lb}<\epsilon_0
		\quad \text{on $[0,l],\;\;$ being $\alpha_\xi:=\exp_p\circ\overline{\alpha}_\xi$.}
		\]
		Moreover, if $\exp_p:\mathcal{D}\to M$ satisfies the continuation property, then $\mathcal{K}$ can be chosen so that $l=b$.
			\end{prop}
	
	\begin{proof} 
Choose $0<\mu<1/C\, b\,\Lambda_{\mathcal{K}}$ and $0<\xi\leq \xi_{\mathcal{K}}$ such that
\begin{equation}\label{ojo}
0<C\,\xi<\mu<
\frac{\delta}
{2\left(\Lambda_{\mathcal{K}}+\frac{1}{L}(1-e^{-Lb})\right)}.
\end{equation}
By Proposition \ref{uff}, 
	there exists 
	a $\xi$-regularized lift
	$\overline{\alpha}_\xi:[0,l]\rightarrow\mathcal{K}$ of $\gamma$, $0<l\leq b$, with $\overline{\alpha}_\xi(0)=v_0$ and $\dot{\overline{\alpha}}(0)=w_0$, such that 
	Im$\,\overline{\alpha}\not\subset \mathring{\mathcal{K}}$
	if $l<b$. 

On the other hand, since $\gamma:[0,b]\rightarrow U\subset M$ is an $h$-geodesic, the function 
	$z^{\gamma}(=z\circ \dot{\gamma})=(x^{\gamma},y^{\gamma})$ satisfies 
	\[
	\frac{dz^{\gamma}}{ds}=F_h(z^{\gamma})=(\pi_1\circ F_h(z^{\gamma}),\pi_2\circ F_h(z^{\gamma}))=(y^{\gamma},\pi_2\circ F_h(z^{\gamma}))\qquad\hbox{on $[0,b]$,}
	\]
	where $\pi_2:{\mathbb R}^{2n}\rightarrow {\mathbb R}^n$ is the projection on the second factor ${\mathbb R}^n$ of ${\mathbb R}^{2n}$. 
	Consider the curve 
	$z^\xi\equiv(x^\xi,y^\xi)$ given by 
	\[
	x^\xi(s):=x^{\alpha_\xi}(s)=x\circ\alpha_\xi(s),\qquad
	y^\xi(s):=y\circ \mathcal{G}^{\,x^\xi}_{0\to s}\big(\dot{\alpha}_\xi(0)\big)\qquad \hbox{on $[0,l]$,}
	\]
	where $\mathcal{G}^{\,x^\xi}_{0\to s}$ denotes the fiberwise $h$-geodesic transport along $x^\xi$, i.e. the map sending $y_0$ to $y(s)$, where $y$ solves
	\[
	\dot y + \Gamma(x^\xi(s))(y,y)=0, \qquad y(0)=y_0.
	\]
	Note that
	\[
		z^\xi(0)=(x\circ \alpha_\xi(0),y\circ \dot{\alpha}_\xi(0))=
		(x\circ\gamma(0),y\circ\dot{\gamma}(0))=z^\gamma(0),
	\] 
	and so
	\[
		\begin{array}{c}
			\|x^{\alpha_\xi}(0)-x^{\gamma}(0)\|_2 =\|x^\xi(0)-x^{\gamma}(0)\|_2=0
			=\|z^\xi(0)-z^{\gamma}(0)\|_2 \\
			=\left(2\mu\int_0^0|\dot{\overline{\alpha}}_\xi(s)|_{h_p}ds\right)e^{L0}+\frac{2\mu}{L}(e^{L0}-1)<\delta e^{Lb}(<\epsilon_0).
		\end{array}
	\]
Observe also that 	\[
	\dot{x}^\xi=\dot{x}^{\alpha_\xi}=y\circ\dot{\alpha}_{\xi},\qquad 
	\dot{y}^\xi=\pi_{2}\circ F_h(z^\xi).
	\]
Therefore,
	\begin{equation}\label{k1}
		\begin{array}{c}
			\|\dot{z}^\xi-F_h(z^\xi)\|_2\leq \|(y\circ\dot{\alpha}_\xi,\dot{y}^\xi) -(y\circ \mathcal{G}^{\,x^\xi}_{0\to s}\big(\dot{\alpha}_\xi(0)\big),\pi_{2}\circ F_h(z^\xi))\|_2 \\
			\leq\|\big(y\circ\dot{\alpha}_\xi- y\circ \mathcal{G}^{\,x^\xi}_{0\to s}\big(\dot{\alpha}_\xi(0)\big),0\big)\|_2=\|y\circ\dot{\alpha}_\xi- y\circ \mathcal{G}^{\,x^\xi}_{0\to s}\big(\dot{\alpha}_\xi(0)\big)\|_2.
		\end{array}
	\end{equation}
	From (\ref{ch}), 
	\begin{equation}\label{k'3}
\begin{array}{c}
		\|y\circ\dot{\alpha}_\xi- y\circ \mathcal{G}^{\,x^\xi}_{0\to s}\big(\dot{\alpha}_\xi(0)\big)\|_2\leq C\,\xi\,|\dot{\overline{\alpha}}_\xi|_{h_p}< \mu|\dot{\overline{\alpha}}_\xi|_{h_p}.
\end{array}
	\end{equation}	
	Thus, putting together (\ref{k1}) and (\ref{k'3}), we deduce 
	\begin{equation}\label{37}
		\|\dot{z}^\xi-F_h(z^\xi)\|_2<
		\mu|\dot{\overline{\alpha}}_\xi|_{h_p}
		\quad\hbox{on $[0,l]$.}
	\end{equation}
	Next, define $\beta(s):=(s,z^\xi(s))$ for all $s\in [0,l]$, and denote by $\hat{G}$ some continuous extension to ${\mathbb R}\times {\mathbb R}^{2n}$ of the vector 
	field $\dot{\beta}(s)=(1,\dot{z}^\xi(s))$ along $\beta\subset {\mathbb R}\times {\mathbb R}^{2n}$. Taking into account (\ref{37}), by the Dugundji extension theorem\footnote{This result generalizes Tietze extension theorem as follows: {\em If $X$ is a metric space, $Y$ is a locally convex topological vector space, $A$ is a closed subset of $X$ and $f:A\rightarrow Y$ is continuous, then it could be extended to a continuous function $\tilde{f}$ defined on all of $X$; moreover, the extension could be chosen such that $\tilde{f}(X)\subset {\rm conv}f(A)$.}} (see \cite{Du}) applied to a sufficiently fine partition of $[0,l]$, $\hat{G}$ can be chosen to additionally satisfy
	\[
	\|\hat{G}(s,z)-(1,F_h)(s,z)\|_2\leq 2\mu(|\dot{\overline{\alpha}}_\xi(s)|_{h_p}+1)\quad\;\forall\,(s,z)\in [0,l]\times {\mathbb R}^{2n}.
	\]
	Denote by $G$ the projection of $\hat{G}$ on the last $2n$ components. Then,
	\[ 
	\begin{array}{c}
		\dot{z}^\xi(s)=G(s,z^\xi(s))\quad\forall s\in [0,l],\quad\hbox{and}
		\\ \|G(s,z^\xi)-F_h(s,z^\xi)\|_2\leq \|\hat{G}(s,z^\xi)-(1,F_h)(s,z^\xi)\|_2\leq 2\mu(|\dot{\overline{\alpha}}_\xi(s)|_{h_p}+1)=\kappa_1(s)+\kappa_2 \\ \hbox{where $\kappa_1(s):=2\mu|\dot{\overline{\alpha}}_\xi(s)|_{h_p}$, $\quad\kappa_2:=2\mu$.}
	\end{array}
	\]
	From Proposition \ref{t} applied to $F_h$ and $G$:
	\[
	\begin{array}{c}
		\|x^{\alpha_\xi}(s)-x^{\gamma}(s)\|_2 
		=\|x^\xi(s)-x^{\gamma}(s)\|_2\leq \|z^\xi(s)-z^{\gamma}(s)\|_2 
		\\   \leq \left(\|z^\xi(0)-z^{\gamma}(0)\|_2+2\mu\int_{0}^{s}|\dot{\overline{\alpha}}_\xi(s')|_{h_p}ds'\right)e^{Ls} + 
		\frac{2\mu}{L}(e^{Ls}-1) 
		\\  
		\leq\left(2\mu\int_0^s|\dot{\overline{\alpha}}_\xi(s')|_{h_p}ds'\right) e^{Ls}+\frac{2\mu}{L}(e^{Ls}-1) \\
		\stackrel{(\ref{rri})}{\leq}2\mu C_{\mathcal{K}} e^{Ls}+\frac{2\mu}{L}(e^{Ls}-1)\stackrel{(\ref{ojo})}{<}\delta e^{Lb}(<\epsilon_0)\quad\hbox{on $[0,l]$}.
	\end{array}
	\]
		
Assume now that $\exp_p$ satisfies the continuation property. By Proposition \ref{1.3} applied to $\exp_p$, there exists a compact neighborhood $\mathcal{C}$ of $v_0$ in $\mathcal{D}$
such that any piecewise smooth regular curve $\overline{\alpha}_\xi:[0,l]\rightarrow \mathcal{D}$, $0<l\leq b$, with $\overline{\alpha}_\xi(0)=v_0$, for which ${\rm length}_h(\alpha_\xi)\leq {\rm length}_h(\gamma)+1$, satisfies 
$\overline{\alpha}_\xi([0,l])\subset \mathcal{C}$. Choose the compact subset $\mathcal{K}\subset \mathcal{D}$ (of the discussion just before this proposition) with $\mathcal{C}\subset \mathring{\mathcal{K}}$. Integrating (\ref{eq:alpha-acc'}) over $[0,s]$, we obtain
\[
|\dot{\alpha}_\xi(s)|_h
\le
|\dot{\alpha}_\xi(0)|_h
+
C\,\xi\int_0^s |\dot{\overline{\alpha}}_\xi(t)|_{h_p}\,dt\leq 
|\dot{\gamma}_\xi(s)|_h
+
C\,\xi\int_0^l |\dot{\overline{\alpha}}_\xi(s)|_{h_p}\,ds,
\]
and integrating again over $[0,l]$,
\[
{\rm length}_h(\alpha_\xi)\stackrel{(\ref{rri})}{\leq} {\rm length}_h(\gamma)+C\,\xi\,l\,C_{\mathcal{K}}\stackrel{(\ref{ojo})}{\leq} {\rm length}_h(\gamma)+1.
\]
So, one must have
	Im$\,\overline{\alpha}_\xi\subset\mathcal{C}\subset \mathring{\mathcal{K}}$, and thus, $l=b$, as required. \end{proof}

\smallskip	
	
	\begin{cor}\label{co} Let $\gamma:I=[0,b]\rightarrow M$ be a geodesic for some Riemannian metric $h$ on $M$, such that $(dºexp_p)(v_0,w_0)=(\gamma(0),\dot{\gamma}(0))$ for some $(v_0,w_0)\in T\mathcal{D}$. Suppose that $\gamma(I)\subset U$ for some relatively compact open coordinate chart $(U,x)$ of $M$.
Then, there exists a partial quasi-lift of $\gamma$ with inital data $(v_0,w_0)$, which is inextensible in $\mathcal{D}$ if it is not global. Moreover, if $\exp_p:\mathcal{D}\rightarrow M$ satisfies the continuation property, then the quasi-lift is global.	
	\end{cor}
	
	\begin{proof} Take a sequence of positive numbers $\{\delta_i\}$ with $\delta_i\rightarrow 0$, and an exhausting sequence $\{\mathcal{K}_i\}$ of compact subsets of $\mathcal{D}$. For each $i$, consider the $\xi_i$-regularized lift $\overline{\alpha}_{\xi_i}:[0,l_i]\rightarrow \mathcal{K}_i$, $i\in \mathbb{N}$, of $\gamma$ provided by Proposition \ref{previous}. We know that 
	\[
	 \|x^{\alpha_{\xi_i}}(s)-x^{\gamma}(s)\|_2<\delta_i e^{Lb}\;\;\hbox{on $[0,l_i]$ with $\alpha_{\xi_i}:=\exp_p\circ\overline{\alpha}_{\xi_i}$,}
	\]
	and thus, we fall under the hypotheses of the Ascoli-Arzelá theorem (in the form given in \cite[Theorem 2.5.14]{Bu}). So, the constant $h_p$-speed parametrizations $\{\tilde{\alpha}_{\xi_i}:[0,1]\rightarrow \mathcal{K}_i\}_{i\in \mathbb{N}}$ admits some limit curve $\overline{\alpha}:[0,1]\rightarrow \mathcal{D}$ satisfying $\gamma\circ\chi = \exp_p\circ\overline{\alpha}$ for some continuous, nondecreasing, surjective function $\chi:[0,1]\rightarrow [0,c]$, with $0<c\leq b$.
	Moreover, if $c<b$ then $l_i< b$ for all $i$ big enough, that is, $\overline{\alpha}_{\xi_i}$, or equivalenty $\tilde{\alpha}_{\xi_i}$, is inextensible in $\mathcal{K}_i$ for all $i$ big enough, which implies that
	$\overline{\alpha}$ is inextensible in $\mathcal{D}$.
	
	Finally, if $\exp_p:\mathcal{D}\rightarrow M$ satisfies the continuation property, by Proposition \ref{previous} we can take $\mathcal{K}\subset \mathcal{D}$ so that $l_i=b$ for all $i$. Then, limit curve $\overline{\alpha}$ obtained by previous procedure for $\{\mathcal{K}_i\equiv \mathcal{K}\}_{i\in \mathbb{N}}$ is necessarily a global quasi-lift. 
	\end{proof}

	\medskip
	
	\noindent {\bf {\em Proof of Theorem \ref{maint}.}} 
	Let $\gamma:[0,1]\rightarrow M$ be a smooth regular curve with $(d\exp_p)(v_0,w_0)=(\gamma(0),\dot{\gamma}(0))$ for some $(v_0,w_0)\in T\mathcal{D}$. 
	Then, there exists a partition $0=b_0< b_1<\cdots < b_{m}=1$, together with relatively compact open coordinate charts $(U_k,x_k)$ of $M$ such that each $\gamma\mid_{[b_k,b_{k+1}]}$ is a simple curve contained in $U_k$, $k=0,\ldots,m-1$. There exist also Riemannian metrics $h_k$, $k=0,\ldots,m-1$, on $M$ such that each $\gamma\mid_{[b_k,b_{k+1}]}$ is an $h_k$-geodesic for $h_k$\footnote{Locally, one can choose a Riemannian metric whose Levi--Civita connection makes the curve geodesic; a global metric is then obtained by extending these local metrics to $M$ using a partition of unity.}. 
	Then, the partial quasi-lift of $\gamma$ is obtained by concatenating the partial quasi-lifts arising from the iterative application of Corollary \ref{co} to each curve $\gamma_k$, for $k=0,\ldots,m-1$, until no further extension is possible, i.e., until an inextendible curve $\mathcal{D}$ is reached.
	
	For the last statement, just repeat previous procedure, but taking into account that the partial quasi-lifts are now global for each $k\in \{0,\ldots,m-1\}$, and so, the iterative procedure can be completed. \qed

{
		\begin{rem}\label{ri} The quasi-lift obtained through this procedure --essentually based on Definition \ref{but}, Proposition \ref{uff}, the Ascoli-Arzelá theorem, and the $h_p$-unitary reparametrization-- depends continuously on the original base curve. That is, if a family of curves on the manifold varies continuously, then their corresponding quasi-lifts on the tangent space vary continuously as well.
		\end{rem}}

	\section{Applications to geodesics on semi-Riemannian manifolds}\label{s5}

	In this section, we derive several immediate consequences of our main result. The effectiveness of our approach becomes particularly evident in light of the fact that these consequences generalize classical theorems ---despite having originally been proved using substantially more involved variational or topological techniques.
	
	We begin with the following straightforward consequence of Theorem \ref{maint}.

	\begin{thm}\label{ñh}
		{\em Let $(M,g)$ be a connected semi-Riemannian manifold, and assume that for some $p\in M$ the exponential map $\exp_p:\mathcal{D}\rightarrow M$ satisfies the continuation property. Then, $\exp_p$  is surjective and thus there exists a geodesic connecting $p$ with any other $q\in M$. Moreover, this geodesic can be chosen to be fixed-endpoint homotopic to any piecewise smooth curve in $M$ joining $p$ and $q$.}
	\end{thm}
	
	\begin{proof}
		Fix $q\in M$ and let $\gamma:[0,1]\rightarrow M$ be any piecewise smooth 
		curve 
		with $\gamma(0)=p$, $\gamma(1)=q$. By Theorem \ref{maint} applied to $\exp_p:\mathcal{D}\subset T_p M\rightarrow M$, there exists a quasi-lift $\overline{\alpha}:[0,c]\rightarrow \mathcal{D}$ of $\gamma$ with $\overline{\alpha}(0)=0_{T_pM}$. Let $v:=\overline{\alpha}(c)\in \mathcal{D}$, and let $\chi:[0,c] \rightarrow [0,1]$ be a 
		continuous, nondecreasing, surjective function
		such that $\gamma\circ \chi = \exp_p\circ \overline{\alpha}$. Then, $\chi(c) =1$, and thus, $$\exp_p(v) = (\exp_p\circ \overline{\alpha})(c) = \gamma(\chi(c)) = \gamma(1)= q.$$ This proves the surjectivity of $\exp_p$. 
		
		For the last assertion, notice that $\mathcal{D}$ is star-shaped around $0_{T_pM}$, hence 1-connected. Thus, the curve $t\in [0,1]\mapsto \exp_p(t\,v)=\Phi_X(1,tv)=\Phi_X(t,v)\in M$ joins $\exp_p(0)=\exp_p\circ\overline{\alpha}(0)=\gamma(0)=p$ with $\exp_p(v)=\exp_p\circ\overline{\alpha}(c)=\gamma(1)=q$. Moreover, since $\mathcal{D}$ is $1$-connected, 
		the segment $t\in [0,1]\mapsto t\, v\in \mathcal{D}$ and
		the curve $\overline{\alpha}$ are endpoint-homotopically equivalent and so are the corresponding  compositions $\eta: t\in [0,1]\mapsto \exp_p(t\, v)\in M$ and
		$\exp_p\circ\overline{\alpha}=\gamma \circ \chi$ (the latter being a nondecreasing reparametrization of $\gamma$). 
			\end{proof}

\smallskip

A classical result by Morse \cite[Thm. 13.3, p. 239]{Morse}, later refined by a key contribution from Serre \cite{Se}, establishes that any two points in a complete, non-contractible Riemannian manifold can be joined by infinitely many geodesics. {Note that, in this case, the corresponding exponential map has the continuation property but is non-proper. The following direct consequence of our approach shows that, in the semi-Riemannian setting, these three properties (completeness, continuation property and non-properness) actually suffices to guarantee the existence of infinitely many connecting geodesics—thus yielding a significant generalization of the classical result.}

\begin{thm}\label{thr}
	Let $(M,g)$ be a semi-Riemannian manifold, and consider some $p\in M$ at which the exponential map $\exp_p:\mathcal{D}\subset T_p M\rightarrow M$ has the continuation property but is non-proper. Then there exist infinitely many geodesics connecting $p$ with any point of $M$ (including $p$ itself). In particular, there exist infinitely many geodesic loops based at $p$.
\end{thm}	

\noindent {\it Proof.} Let $q\in M$ arbitrary. Since $\exp_p$ is not proper, there exists a compact set $K\subset M$ such that the preimage $\exp_p^{-1}(K)\subset \mathcal{D}$ is non-compact. Hence, one can choose a sequence $\{u_n\}_{n\in\mathbb{N}}\subset \exp_p^{-1}(K)$ that is not contained in any compact subset of $\mathcal{D}$. 
For each $n$, let $\beta_n$ be a smooth curve in $M$ connecting $r_n:=\exp_p(u_n)\in K$ to $q$, 
with length$_h(\beta_n)<\mathcal{L}$ for some $\mathcal{L}>0$ and for some auxiliary complete Riemannian metric $h$ on $M$.
By Theorem \ref{maint}, each curve $\beta_n$ admits a quasi-lift $\overline{\beta}_n\subset \mathcal{D}$ starting at $u_n$ and ending at a point $v_n\in \exp_p^{-1}(q)$. Define the geodesic $\gamma_n(s):=\exp_p(s\cdot v_n)$, which connects $p$ to $\exp_p(v_n)=q$. To complete the argument, it suffices to show that the sequence $\{v_n\}_{n\in\mathbb{N}}\subset \mathcal{D}$ has no convergent subsequences. Assume, for contradiction, that a subsequence $\{v_{n_k}\}_{k\in\mathbb{N}}$ converges to some $v\in \mathcal{D}$. 
Since the family of quasi-lifts $\{\overline{\beta}_{n_k}\}$ is not contained in any compact subset of $\mathcal{D}$, by extending each curve with a short segment joining $v_{n_k}$ to $v$ we obtain a new family of curves whose compositions with $\exp_p$ still have length less than $\mathcal{L}+1$, thereby contradicting the continuation property of $\exp_p$ stated in Proposition~\ref{1.3}\,(iii). \qed

\medskip

\subsection{Riemannian results}

We give here a connectedness through \textit{minimizing} geo\-de\-sics result that is independent of Hopf-Rinow's classic arguments. 
\begin{thm}\label{tuss}
	Let $(M,g)$ be a connected Riemannian manifold. If the exponential map $\exp_p:\mathcal{D}\subset T_pM \rightarrow
	M$ has the continuation property for some (resp. any) $p\in M$, then there exists a {\em minimizing} geodesic connecting $p$ with any other $q\in M$ 
	(resp. connecting any pair of points in $M$).
\end{thm}

\noindent {\it Proof.} Assume that $\exp_p:\mathcal{D}\subset T_pM \rightarrow
M$ has the continuation property for some $p\in M$, and fix any other $q\in M$. Let $\{\gamma_n\}$ be a sequence of smooth curves $\gamma_n:[0,1]\rightarrow M$ joining $p$ with $q$ such that length$(\gamma_n)\rightarrow d(p,q)$. By Theorem \ref{ñh}, there exists a quasi-lift $\overline{\alpha}_n:[0,b_n]\rightarrow \mathcal{D}$ of $\gamma_n$ with $\overline{\alpha}_n(0)=0$ for each $n$. Clearly, $\exp_p\circ\overline{\alpha}_n(b_n)=q$ for all $n$. 
By the continuation property, the sequence $\{\overline{\alpha}_n(b_n)\}$ is contained in a compact set of $\mathcal{D}$ (Proposition \ref{1.3} (iii)). Let $v\in \mathcal{D}$ be a limit (up to a subsequence) of it. By continuity, $$\exp_p(v)=\exp_p\left(\lim_n\overline{\alpha}_n(b_n)\right)=\lim_n \exp_p(\overline{\alpha}_n(b_n))=q,$$ and thus, the geodesic $\gamma:[0,1]\rightarrow M$ given by $\gamma(t):=\exp_p(tv)$ satisfies $\gamma(0)=p$ and $\gamma(1)=q$. 
Moreover, by the Gauss Lemma, ${\rm length}(\gamma)\leq{\rm length}(\exp_p\circ\overline{\alpha}_n)$, hence $$(d(p,q)\leq){\rm length}(\gamma)\leq\lim_n{\rm length}(\exp_p\circ\overline{\alpha}_n)=\lim_n{\rm length}(\gamma_n)=d(p,q),$$
as required. \qed

\begin{rem}\label{rf}
	The exponential map $\exp_p:\mathcal{D}(=T_pM)\rightarrow M$ on a \textit{complete} Riemanninan manifold $(M,g)$ satisfies the continuation property  (see \cite[Prop. 2.6]{CF}). Thus, as pointed out before, the geodesic connectedness statement of the Hopf-Rinow Theorem can be seen as a particular consequence of Theorem \ref{tuss}.
\end{rem}

Let $(M,g)$ be a Riemannian manifold. Denote by $\overline{M}$ (resp. $\partial M$) the {\em Cauchy completion} (resp. {\em boundary}) associated to the metric space $(M,d)$, where $d=d_g$ is the distance function on $M$ associated with $g$. A curve $\sigma:[a,b)\rightarrow M$ joins $\sigma(a)=p\in M$ with $q\in \partial M$ if the extension $\overline{\sigma}:[a,b]\rightarrow \overline{M}$ of $\sigma$ defined by imposing that $\overline{\sigma}(b):=q$ is continuous in $\overline{M}$. The notion of a curve $\sigma:(a,b)\rightarrow M$ joining two points $p,q\in\partial M$ is defined analogously. With these definitions we can now establish the following extension of the previous result.
\begin{thm}\label{tus}
	Let $(M,g)$ be a Riemannian manifold. If the exponential map $\exp_p:\mathcal{D}\subset T_pM \rightarrow
	M$ has the continuation property for some (resp. any) $p\in M$, then there exists a {\em minimizing} geodesic connecting $p$ with any other $q\in \overline{M}$ (resp. connecting any pair of points in $\overline{M}$).
\end{thm}

\noindent {\it Proof.} Assume that the map $\exp_p:\mathcal{D}\subset T_pM \rightarrow
M$ has the continuation property for $p\in M$, and suppose that $q\in\partial M$ (the other case is similarly obtained). Let $\{\gamma_n\}$ be a sequence of smooth curves $\gamma_n:[0,1]\rightarrow M$ such that $\gamma_n(0)=p$, $\{\gamma_n(1)\}$ converges to $q$ in $\overline{M}$, and ${\rm length}(\gamma_n)\rightarrow d(p,q)$.
By Theorem \ref{ñh}, there exists a quasi-lift $\overline{\alpha}_n:[0,b_n]\rightarrow \mathcal{D}$ of $\gamma_n$ with $\overline{\alpha}_n(0)=0$ for each $n$.
Consider the curves $t\in [0,1]\mapsto \exp_p(t\overline{\alpha}_n(b_n))$. By the Gauss Lemma we have 
\begin{equation}\label{xu}
	{\rm length}(t\mapsto \exp_p(t\overline{\alpha}_n(b_n))\leq{\rm length}(\gamma_n).
\end{equation}
Taking into account that  $\{{\rm length}(\gamma_n)\}$ is bounded, we deduce that $\{\overline{\alpha}_n(b_n)\}\subset \mathcal{D}$ is contained in a compact set of $T_p M$. Let $v\in \overline{\mathcal{D}}$ be a limit  of it (up to a subsequence). By continuity, the geodesic $\gamma:[0,1)\rightarrow M$ given by $\gamma(t):=\exp_p(tv)$ satisfies $\gamma(0)=p$ and $\lim_{t\rightarrow 1}\gamma(t)=\lim_{t\rightarrow 1}\exp_p(tv)=q$. 
In conclusion, 
\[
(d(p,q)\leq){\rm length}(\gamma)=\lim_n {\rm length}(t\mapsto \exp_p(t\overline{\alpha}_n(b_n))\stackrel{(\ref{xu})}{\leq}
\lim_n{\rm length}(\gamma_n)=d(p,q),
\]
as required. \qed

\medskip

We ends with the following direct consequence of 
Theorem \ref{thr} and Remark \ref{rf}: 

\begin{cor}\label{AA}
	Let $(M,g)$ be a complete Riemannian manifold whose exponential map is non-proper at every point (which happens, for instance, when $M$ is compact or non-contractible). Then, there exist infinitely many geodesics connecting any point $p$ with any point of $M$ (including $p$ itself). In particular, there exist infinitely many geodesic loops based at each point of $M$.
\end{cor}
\medskip

\subsection{Lorentzian results} $\quad$ \newline

Certainly, the semi-Riemannian results presented at the beginning of this section can be adapted to the Lorentz case, providing a valuable contribution in this context. In this subsection, however, we will derive alternative results in terms of causal/timelike curves, due to their physical implications in the context of general relativity. 

For each $p\in M$ we denote by $\mathcal{T}_p \subset T_pM$ the set of timelike vectors at $p$. Recall that $\mathcal{T}_p$ is the disjoint union of two connected open convex cones called {\it timecones}. 

We denote by $\mathcal{C}_p$ the closure of $\mathcal{T}_p$ in $T_pM$. Note that $0_p \in \mathcal{C}_p$, and that $\mathcal{C}_p \setminus \{0_p\}$ coincides with the set of causal vectors in $T_pM$. Again, $\mathcal{C}_p\setminus \{0_p\}$ has two connected components called {\it causal cones}. A piecewise smooth curve $\sigma:[a,b]\rightarrow M$ is said to be {\it timelike} [resp. {\it causal}] if its tangent vector $\sigma '(t) \in \mathcal{T}_{\sigma(t)}$ [resp. $\in \mathcal{C}_{\sigma(t)}\setminus \{0_{\sigma(t)}\}$] for any $t\in [a,b]$ and both lateral tangent vectors at a break are on the same component of the timecone [resp. causal cone] thereat.

Let
\begin{equation}\label{causalBigC}
	C_p:= \mathcal{C}_p\cap \mathcal{D}.
\end{equation}
Following standard notation, we write
\begin{align}
	I(p) &=\{q\in M \, : \, \hbox{$\exists$ a piecewise smooth timelike segment connecting $p$ and $q$}\}, \notag \\
	J(p)&= \{q\in M \, : \, \hbox{$\exists$ a piecewise smooth causal segment connecting $p$ and $q$}\}\cup\{p\}.\notag
\end{align}
It is well-known that $I(p)$ is always open.

\begin{defi}[Causal continuation property]\label{def1}
	Let $p\in M$. We say that $\exp _p$ has the {\it causal continuation property} (CCP) if for any (piecewise smooth) causal curve $\gamma:[0,1] \rightarrow M$ with $\gamma(0)=p$, and for any continuous curve $\sigma:[0,b)\subset [0,1] \rightarrow C_p$ [for $C_p$ defined in (\ref{causalBigC})] such that $\sigma(0)=0_p$ and 
	\[\exp_p\circ \sigma = \gamma\mid_{[0,b)}\]
	there exists a sequence $(t_k)_{k \in \mathbb{N}}\subset [0,b)$ with $t_k\rightarrow b$ for which $\{\sigma(t_k)\}_{k\in\mathbb{N}}$ converges in $\mathcal{D}$ (and thus, in $C_p$). 
\end{defi}

According to Theorem \ref{maint}, if $\gamma:[0,1] \rightarrow M$ is a (piecewise smooth) causal curve which does not admit a (global) quasi-lift then there exists a curve $\overline{\alpha}$ in $C_p$ that is inextensible in $\mathcal{D}$, thus violating the CCP. Consequently:
\begin{cor}\label{maintt} 
	Let $(M,g)$ be a Lorentzian manifold, and assume that the exponential map $\exp_p$ has the CCP for some $p\in M$. Then,
	any (piecewise smooth) causal curve $\gamma:[0,1] \rightarrow M$ with $\gamma(0)=p$ admits a quasi-lift $\overline{\alpha}:[0,c]\rightarrow C_p$ starting at $0_p\in C_p$.
\end{cor}

The following theorem aims at giving sufficient conditions to ensure the existence of a maximizing {\it causal} geodesic from $p \in M$ to $q\in J(p)$. 

\begin{thm}\label{thmclave3}
	Let $(M, g)$ be a Lorentzian manifold, and assume that $\exp _p$ has the CCP for some $p \in M$. If  there exists a causal curve $\gamma$ from $p$ to $q$,  
	then there exists a maximizing causal geodesic from $p$ to $q$.  
	In particular, if $p=q$ so that
	$\alpha$ is a timelike loop, then there exists a timelike geodesic loop $\gamma$ at $p$. 
\end{thm}
\begin{proof} 
	Suppose a causal curve $\gamma$ exists connecting $p$ to $q$. If $q \in J(p) \setminus I(p)$ then $\gamma$ itself can be reparametrized as a {\it null} geodesic segment connecting $p$ with $q$ (cf. \cite[Prop. 10.46]{O}). So, we will focus on the case $q \in I(p)$.
	
	From Corollary \ref{maintt}, given a sequence of timelike curves $\{\gamma_n:[0,1]\rightarrow M\}$ joining $p$ with $q$ such that length$(\gamma_n)\rightarrow d(p,q)$, there exist quasi-lifts $\overline{\alpha}_n:[0,c_n] \rightarrow C_p$ of $\gamma_n$ with $\overline{\alpha}_n(0)=0_p$ for each $n$.  
	The key observation here is that since $\exp_p \circ \overline{\alpha}_n$ is timelike by contruction, by \cite[Lemma 5.33]{O}, we have $\overline{\alpha}_n\subset \mathcal{T}_p$ (and indeed $\overline{\alpha}_n$ stays within a single timecone). Moreover, $\exp_p(\overline{\alpha}_n(c_n)) = q$ for all $n$. By the CCP, the sequence $\{\overline{\alpha}_n(c_n)\}$ is contained in a compact set of $C_p$. Let $v\in C_p$ ve a limit (up to a subsequence) of it. By continuity, $$\exp_p(v)=\exp_p\left(\lim_n\overline{\alpha}_n(c_n)\right)=\lim_n\exp_p(\overline{\alpha}_n(c_n))=q,$$ and thus, the geodesic $\gamma:[0,1]\rightarrow M$ given by $\gamma(t):=\exp_p(tv)$ satisfies $\gamma(0)=p$ and $\gamma(1)=q$. 
	Moreover, by Gauss Lemma, ${\rm length}(\gamma_n)\geq{\rm length}(\exp_p\circ\overline{\alpha}_n)$, hence $$(d(p,q)\geq){\rm length}(\gamma)\geq\lim_n{\rm length}(\exp_p\circ\overline{\alpha}_n)=\lim_n{\rm length}(\gamma_n)=d(p,q),$$
	as required.
\end{proof}
{\it The hypothesis of causal continuation cannot be removed in Theorem \ref{thmclave3}}. To see this, just consider the flat Lorentzian manifold $(M:=\mathbb{R}^2\setminus \{(1,0)\},-dt^2+dx^2)$, $p=(0,0), q=(2,0)$. Then $q\in I(p)$, but there is no timelike geodesic connecting them. Indeed, 
\[\mathcal{D} \equiv \mathbb{R}^2 \setminus \{(t,0)\, : \, t\geq 1\}.\]
Given any timelike curve $\sigma:[0,1] \rightarrow M$ from $p$ to $q$, its portion $\sigma\mid_{[0,1)}$ admits a lift to $C_p$ through $\exp _p$, but it cannot be extended in $\mathcal{D}$.

\medskip

It is well-known (cf., e.g., \cite[Prop. 7.36]{BEE}) that if a spacetime
$(M,g)$ is globally hyperbolic, then it is causally pseudoconvex\footnote{For any compact set $K\subset M$ there exists a compact set
	$K^*\subset M$ such that any segment of a causal geodesic with endpoints in $K$ is
	entirely contained in $K^*$.} and disprisoning\footnote{For any given maximal extension $\gamma:(a,b)\rightarrow M$ of a causal geodesic
	($-\infty \leq a < b \leq \infty$), and any $t_0\in (a,b)$, neither $\gamma[t_0, b)$ nor $\gamma(a, t_0]$
	is compact.}. On the other hand, if $(M,g)$ is
causally pseudoconvex and disprisoning, then $\exp_p\mid_{C_p}$ is a proper map (that is, inverse images of compact sets are compact) for every $p\in M$ (see \cite[Corollary 3.6]{CFH}), and consequently, it has the CCP. In conclusion:
\begin{prop}\label{p2} If $(M,g)$ is a globally hyperbolic spacetime then $\exp_p$ has the CCP for any $p\in M$.
\end{prop}

\medskip

\noindent As a consequence of Theorem \ref{thmclave3} and Proposition \ref{p2}, we obtain another proof of the following well-known classic result:
\begin{cor} (Avez-Seifert).
	Let $(M,g)$ be a globally hyperbolic spacetime. If $p< q$ then there exists a future-directed maximizing causal geodesic connecting $p$ with $q$.
\end{cor}

\section{Beyond the exponential map: abstract lifting frameworks}\label{abstractlifting}

The theory developed in this work—centered on quasi-lifting of curves via the exponential map—extends well beyond the classic semi-Riemannian setting. Indeed, the core lifting phenomenon, as well as the path-continuation principle and compactness arguments near the singular strata, rely only on broad aspects of the local structure of the geodesic flow-induced map on the tangent bundle, together with certain topological properties of its domain. 

More precisely, the entire framework applies with almost no change to any map 
\[
\exp_X : \mathcal{D}_X \subset TM \to M
\]
associated with a smooth vector field \( X \in \mathfrak{X}(TM) \), as long as the following 
mild condition is satisfied.

\noindent (*) {\em Star-shaped fibers}: for each \( p \in M \), the domain \( \mathcal{D}_p := \mathcal{D}_X \cap T_pM \) is star-shaped with respect to the origin.

A natural general condition that is shared by the geodesic flow and which can ensure (*) is that the vector field $X:TM \rightarrow TTM$ satisfies the condition 
$$d\pi_v(X_v) = v \quad \forall v \in TM,$$
where $\pi:TM \rightarrow M$ is the standard projection. In that case, let $\Phi_X: \mathcal{O}_X\subset \mathbb{R}\times TM \rightarrow TM$ be the global flow of $X$, and consider the open set
\[
\mathcal{D}_X := \{v \in TM \, : \, (1,v) \in \mathcal{O}_X\}.
\]
We can then define
\[
\exp_{X} : \mathcal{D}_{X} \subset TM \to M, \quad \exp_{X}(v) := \pi \circ \Phi_{X}(1, v), 
\]
in complete analogy with the exponential map arising from a semi-Riemannian geodesic spray. This construction provides a natural generalization of the classical exponential map: indeed, when 
\(X = X_g\) is the geodesic spray associated with a semi-Riemannian metric \(g\), the map 
\(\exp_{X_g}\) coincides with the usual exponential map of \((M,g)\).

The curves of the form
$$\gamma_v(t) := \pi\circ \Phi_X(t,v) = \exp_X(t\cdot v)$$
on $M$ for $v\in D_X$ play the role of geodesics, and indeed  can easily be seen to satisfy a system of semi-linear second-order equations 
$$\frac{d^2(x^i\circ \gamma_v)}{dt^2} = V^i_j\left(x^j\circ \gamma_v(t),\frac{d(x^j\circ \gamma_v)}{dt}\right), \quad i=1,\ldots, n$$
in local coordinates $(x^1,\ldots, x^n)$ on $M^n$. 

Condition (*) then follows from the local solvability and uniqueness of solutions to second-order ODEs, which ensure that the flow domain \( \mathcal{D}_p \) around the zero vector in \( T_pM \) is open and star-shaped. 

It is worth emphasizing that one of the key tools traditionally used to study the local and global geometry of geodesic flows---namely, Jacobi fields---also admits a natural extension to this more general setting. Given a second-order system as above, one can define a corresponding variational equation along any solution curve, governing the behavior of infinitesimal variations through nearby trajectories. These generalized Jacobi fields arise as solutions to the linearization of the second-order flow and retain much of the structural information familiar from the classical theory: in particular, they allow the identification of conjugate points, describe local rigidity phenomena, and play a central role in understanding the stratified behavior of the domain of the flow map. Therefore, they can be a powerful analytical and geometric tool in the abstract lifting framework developed here.

These observations indicate that the techniques introduced here form the basis for a general geometric theory of curve lifting via flow-induced maps, independently of any particular metric structure. It reinforces the central idea of the paper, that path-lifting formulated in terms of quasi-lifts together with a suitable flow structure, is a general topological mechanism that transcends the specific geometry of the exponential map. We believe that further development of this framework could yield new insights into geodesic dynamics, generalized connection theories, and topological control problems in singular geometries, ultimately laying the groundwork for a unified approach to lifting phenomena in diverse geometric contexts.

\section*{Acknowledgements}
The authors are partially supported by the projects PID2020-118452GBI00 and PID2024-156031NB-I00. JLF is also partially supported by the IMAG-María de Maeztu grant CEX2020-001105-M (funded by MCIN/AEI/10.13039\-/50110001103).

\end{document}